\documentclass{article}

\usepackage[french,british]{babel}

\usepackage[T1]{fontenc}

\usepackage{amsmath}

\usepackage{amsfonts}
\usepackage{amsbsy}
\usepackage{amssymb}

\usepackage{amsthm}

\theoremstyle{plain}

\usepackage{latexsym}
\usepackage{amsmath}
\usepackage{amsfonts}
\usepackage{amssymb}
\usepackage{psfrag}
\usepackage{graphicx}

\usepackage{amsmath,amssymb}
\usepackage{graphicx}

\usepackage{amsmath,amssymb}
\usepackage{graphicx}

\newtheorem{ozn}{Definition}[section]
\newtheorem{thm}{Theorem}[section]

\newtheorem{zau}{Remark}[section]
\newtheorem{lema}{Lemma}[section]

\newtheorem{remark}{Remark}[section]

\newcommand{\me}{\mathbf}
\newcommand{\mr}{\mathbb}
\newcommand{\mt}{\mathsf}
\newcommand{\md}{\mathcal}
\newcommand{\ld}{\left}
\newcommand{\rd}{\right}

\newcommand{\ip}{\int_{-\pi}^{\pi}}

\newcommand{\be}{\begin{equation}}
\newcommand{\ee}{\end{equation}}

\newcommand{\bem}{\begin{multline}}
\newcommand{\eem}{\end{multline}}

\newcommand{\bml}{\begin{multline*}}
\newcommand{\eml}{\end{multline*}}

\newcommand{\beg}{\begin{gather}}
\newcommand{\eeg}{\end{gather}}

\begin{document}

\title{On filtering problem for stochastic processes with periodically correlated increments}

\author{
Maksym Luz\thanks {BNP Paribas Cardif in Ukraine, Kyiv, Ukraine, maksym.luz@gmail.com},
Mikhail Moklyachuk\thanks
{Department of Probability Theory, Statistics and Actuarial
Mathematics, Taras Shevchenko National University of Kyiv, Kyiv 01601, Ukraine, moklyachuk@gmail.com}
}

\date{\today}

\maketitle

\renewcommand{\abstractname}{Abstract}
\begin{abstract}

We deal with the problem of the mean square optimal estimation of linear transformations of the unobserved values of
a continuous time  stochastic process with periodically correlated increments.
Estimates are based on observations of the process with a continuous time  stochastic noise process which is periodically correlated increments as well.
To solve the problem, we transform the processes  to  infinite dimensional vector valued stationary sequences.
We obtain formulas  for calculating  the mean square errors  and the spectral characteristics  of the optimal estimates  of the transformations.
Formulas determining  the least favorable spectral densities and the minimax-robust spectral characteristics  of the optimal estimates  of transformations are derived.
\end{abstract}

\maketitle

\textbf{Keywords}: {Periodically Correlated Increments, Minimax-Robust Estimate, Mean Square Error}

\maketitle

\vspace{2ex}
\textbf{\bf AMS 2010 subject classifications.} Primary: 60G10, 60G25, 60G35, Secondary: 62M20, 62P20, 93E10, 93E11

\section*{Introduction}

In this article, we deal with a filtering problem in the case where both  a signal $\xi(t)$, $t\in\mr R$, and a noise $\eta(t)$, $t\in\mr R$,  are two uncorrelated continuous time stochastic processes    with periodically
stationary (cyclostationary, periodically correlated) increments $\xi^{(d)}(t,\tau T)=\Delta_{T\tau}^d \xi(t)$ and $\eta^{(d)}(t,\tau T)=\Delta_{T\tau}^d \eta(t)$ of order $d$ and period $T$, where $\Delta_{s} \xi(t)=\xi(t)-\xi(t-s)$.

The classical estimation methods of solution  of interpolation, extrapolation (prediction) and filtering problems for stationary stochastic sequences and processes
are developed by Kolmogorov \cite{Kolmogorov},  Wiener \cite{bookWien}, and  Yaglom \cite{Yaglom}.
 
The concept of stationarity   admits some generalizations, a combination of two of which --   stationary $d$th increments and periodical  correlation -- is in scope of this article.
Random processes with stationary $d$th increments $x(t)$ were introduced   by Yaglom and Pinsker \cite{Pinsker}.
They described the spectral representation of such process and solved the extrapolation problem for these processes.
 The minimax-robust extrapolation, interpolation and filtering  problems for stochastic  processes with stationary increments were investigated by Luz and Moklyachuk \cite{Luz_Mokl_book}.

Dubovetska and Moklyachuk \cite{Dubovetska6}  derived the  classical and minimax-robust estimates for another generalization of stationary processes --  periodically  correlated (cyclostationary) processes, introduced by Gladyshev \cite{Glad1963}.
Periodically  correlated  processes are widely used in signal processing and communications (see Napolitano \cite{Napolitano} for a review of recent works on cyclostationarity and its applications).

In this article, we deal with the problem of the mean-square optimal estimation of the linear transformation  $A\xi=\int_{0}^{\infty}a(t)\xi(-t)dt$
which depend on the unobserved values of a continuous time stochastic process $\xi(t)$ with periodically stationary $d$th
increments. The noise stochastic process $\eta(t)$ is uncorrelated with $\xi(t)$ and has periodically stationary increments as well.
Observations $\zeta(t)=\xi(t)+\eta(t)$ are available at points $t\leq0$.

Similar problems for discrete time processes have been studied by Kozak and Moklyachuk \cite{Kozak_Mokl}, Luz and Moklyachuk \cite{luz3,Luz_Mokl_filtr_GMI,Luz_Mokl_filtr_PCI}.
The extrapolation problem for continuous time stochastic process $\xi(t)$ with periodically stationary $d$th increments based on  observations   without noise
was introduced and studied by Luz and Moklyachuk \cite{Luz_Mokl_extra_cont_PCI}.

The article is organized as follows.
   In section \ref{section_PCI_c}, we describe a presentation of a continuous time stochastic  process with the periodically stationary increments as a  stationary H-valued increment sequence.
    The traditional Hilbert space projection method of filtering is developed in section \ref{classical_filtering_f_n-n_c}.
    Particularly, formulas for calculating  the mean-square error and the spectral characteristic of the optimal linear
estimates of the functional $A\xi$  are derived under some conditions on spectral densities. It's also shown that this method can't be applied for the functional $A_{NT}\xi$ with a finite interval of integration in a general case due to strict conditions on the function $a(t)$, $t\geq0$ (which is not the case for periodically correlated noise).
An approach to solution of the filtering problem which is based on factorizations of spectral densities is developed
in section \ref{classical_filtering_fact_f_n-n_c}.
In section \ref{minimax_filtering} we present our results on minimax-robust filtering for the studied processes:
relations that determine the least favourable spectral densities and the minimax spectral characteristics are derived for some classes of spectral densities.

\section{Stochastic processes with periodically correlated $d$th increments}
\label{section_PCI_c}

\begin{ozn}[Gladyshev \cite{Glad1963}] A mean-square continuous stochastic process
$\eta:\mathbb R\to H=L_2(\Omega,\mathcal F,\mathbb P)$, with $\textsf {E} \eta(t)=0,$ is called periodically correlated (PC) with period  $T$, if its correlation function  $K(t,s)={\textsf{E}}{\eta(t)\overline{\eta(s)}}$  for all  $t,s\in\mathbb R$ and some fixed $T>0$ is such that
\[
K(t,s)={\textsf{E}}{\eta(t)\overline{\eta(s)}}={\textsf{E}}{\eta(t+T)\overline{\eta(s+T)}}=K(t+T,s+T).
\]
\end{ozn}

For a given stochastic process $\{\xi(t),t\in\mathbb R\}$, consider the stochastic $d$th increment process
\begin{equation}
\label{oznachPryrostu_cont}
\xi^{(d)}(t,\tau)=(1-B_{\tau})^d\xi(t)=\sum_{l=0}^d(-1)^l{d \choose l}\xi(t-l\tau),
\end{equation}
with the step $\tau\in\mr R$, generated by the stochastic process $\xi(t)$. Here $B_{\tau}$ is the backward shift operator: $B_{\tau}\xi(t)=\xi(t-\tau)$, $\tau\in \mr R$.
We prefer to use the notation $\xi^{(d)}(t,\tau)$ instead of widely used $\Delta_{\tau}^{d}\xi(t)$ to avoid a duplicate with the mean square error notation.

\begin{ozn}[Luz and Moklyachuk \cite{Luz_Mokl_extra_cont_PCI}]
\label{OznPeriodProc2_cont}
A stochastic process $\{\xi(t),t\in\mathbb R\}$ is called a  stochastic
process with periodically stationary (periodically correlated) increments  with the step $\tau\in\mr Z$ and the period $T>0$ if the mathematical expectations
\begin{align*}
\mt E\xi^{(d)}(t+T,\tau T) & =  \mt E\xi^{(d)}(t,\tau T)=c^{(d)}(t,\tau T),
\\
\mt E\xi^{(d)}(t+T,\tau_1 T)\xi^{(d)}(s+T,\tau_2 T)
& =  D^{(d)}(t+T,s+T;\tau_1T,v_2T)
\\
& =  D^{(d)}(t,s;\tau_1T,\tau_2T)
\end{align*}
exist for every  $t,s\in \mr R$, $\tau_1,\tau_2 \in \mr Z$ and for some fixed $T>0$.
\end{ozn}

The functions $c^{(d)}(t,\tau T)$ and  $D^{(d)}(t,s;\tau_1T,\tau_2 T)$ from the Definition \ref{OznPeriodProc2_cont} are called the \emph{mean value} and  the \emph{structural function} of the stochastic
process $\xi(t)$ with periodically stationary (periodically correlated) increments.

\begin{zau}
For spectral properties of one-pattern increment sequence $\chi_{\mu,1}^{(n)}(\xi(m)):=\xi^{(n)}(m,\mu)=(1-B_{\mu})^n\xi(m)$
see, e.g., \cite{Luz_Mokl_book}, p. 1-8; \cite{Gihman_Skorohod}, p. 48--60, 261--268; \cite{Yaglom}, p. 390--430.
The corresponding results for continuous time increment process $\xi^{(n)}(t,\tau)=(1-B_{\tau})^n\xi(t)$ are described in \cite{Yaglom:1955}, \cite{Yaglom}.

For a review of the properties of periodically correlated processes, we refer to \cite{DubovetskaMoklyachuk2013}, \cite{MoklyachukGolichenko2016}. Here we present a generalization of these properties on the processes with periodically correlated increments from Definition \ref{OznPeriodProc2_cont}.
\end{zau}

For the stochastic process $\{\xi(t), t\in \mathbb R\}$ with periodically correlated  increments $\xi^{(d)}(t,\tau T)$ and the integer step $\tau$,   construct
a sequence of stochastic functions
\begin{equation} \label{xj}
\{\xi^{(d)}_j(u):=\xi^{(d)}_{j,\tau}(u)=\xi^{(d)}_j(u+jT,\tau T),\,\, u\in [0,T), j\in\mathbb Z\}.
\end{equation}

Sequence (\ref{xj}) forms a $L_2([0,T);H)$-valued
stationary increment sequence $\{\xi^{(d)}_j,j\in\mathbb Z\}$ with the structural function
\begin{align*}
B_{\xi^{(d)}}(l,j)&= \langle\xi^{(d)}_l,\xi^{(d)}_j\rangle_H =\int_0^T
\textsf{E}[\xi^{(d)}_j(u+lT,\tau_1 T)\overline{\xi^{(d)}_j(u+jT,\tau_2 T)}]du
\\&=\int_0^T D^{(d)}(u+(l-j)T,u;\tau_1T,\tau_2T) du =
 B_{\xi^{(d)}}(l-j).
 \end{align*}
Making use of the orthonormal basis
\[
\{\widetilde{e}_k=\frac{1}{\sqrt{T}}e^{2\pi
i\{(-1)^k\left[\frac{k}{2}\right]\}u/T}, k=1,2,3,\dots\}, \quad
\langle \widetilde{e}_j,\widetilde{e}_k\rangle=\delta_{kj},
\]
 the stationary increment sequence  $\{\xi^{(d)}_j,j\in\mathbb Z\}$ can be represented in the form
\begin{equation} \label{zeta}
\xi^{(d)}_j= \sum_{k=1}^\infty \xi^{(d)}_{kj}\widetilde{e}_k,\end{equation}
where
\[\xi^{(d)}_{kj}=\langle\xi^{(d)}_j,\widetilde{e}_k\rangle =
\frac{1}{\sqrt{T}} \int_0^T \xi^{(d)}_j(v)e^{-2\pi
i\{(-1)^k\left[\frac{k}{2}\right]\}v/T}dv.\]
We call this sequence
$\{\xi^{(d)}_j,j\in\mathbb Z\},$
   or the corresponding to it vector sequence
\be \label{generted_incr_seq}
  \{\vec\xi^{(d)}(j,\tau)=\vec\xi^{(d)}_j=(\xi^{(d)}_{kj}, k=1,2,\dots)^{\top}=(\xi^{(d)}_{k}(j,\tau), k=1,2,\dots)^{\top},
j\in\mathbb Z\},\ee
\emph{an infinite dimension vector stationary increment sequence} generated by the increment process $\{\xi^{(d)}(t,\tau T),t\in\mathbb R\}$.
Further, we will omit the word vector in the notion generated vector stationary increment sequence.

Components $\{\xi^{(d)}_{kj}\}: k=1,2,\dots;j\in\mathbb Z$  of the generated stationary increment sequence
$\{\xi^{(d)}_j,j\in\mathbb Z\}$ are such that, \cite{Kallianpur}, \cite{Moklyachuk:1981}
\[
\textsf{E}{\xi^{(d)}_{kj}}=0, \quad \|\xi^{(d)}_j\|^2_H=\sum_{k=1}^\infty
\textsf{E}|\xi^{(d)}_{kj}|^2\leq P_{\xi^{(d)}}=B_{\xi^{(d)}}(0),
\]
and
\[\textsf{E}\xi^{(d)}_{kl}\overline{\xi^{(d)}_{nj}}=\langle
R_{\xi^{(d)}}(l-j;\tau_1,\tau_2)\widetilde{e}_k,\widetilde{e}_n\rangle.
\]
The \emph{structural function}  $R_{\xi^{(d)}}(j):=R_{\xi^{(d)}}(j;\tau_1,\tau_2)$  of the generated stationary increment
sequence $\{\xi^{(d)}_j,j\in\mathbb Z\}$
 is a correlation operator function.
 The correlation operator $R_{\xi^{(d)}}(0)=R_{\xi^{(d)}}$ is a
kernel operator and its kernel norm satisfies the following limitations:
\[
\|\xi^{(d)}_j\|^2_H=\sum_{k=1}^\infty \langle R_{\xi^{(d)}}
\widetilde{e}_k,\widetilde{e}_k\rangle\leq P_{\xi^{(d)}}. \]

 Suppose that  the structural function $R_{\xi^{(d)}}(j)$ admits a representation
\[
\langle R_{\xi^{(d)}}(j;\tau_1, \tau_2)\widetilde{e}_k,\widetilde{e}_n\rangle=\frac{1}{2\pi} \int _{-\pi}^{\pi}
e^{ij\lambda}(1-e^{-i\tau_1\lambda})^d(1-e^{i\tau_2\lambda})^d\frac{1}
{\lambda^{2d}}\langle f(\lambda) \widetilde{e}_k,\widetilde{e}_n\rangle d\lambda.
\]
Then
$f(\lambda)=\{f_{kn}(\lambda)\}_{k,n=1}^\infty$ is a \emph{spectral density function} of the generated stationary increment sequence $\{\xi^{(d)}_j,j\in\mathbb Z\}$. It  is a positive valued operator  functions of variable
 $\lambda\in [-\pi,\pi)$, and for almost all   $\lambda\in [-\pi,\pi)$ it is a kernel operator with an integrable kernel norm
\begin{multline} \label{P1}
\sum_{k=1}^\infty \frac{1}{2\pi} \int _{-\pi}^\pi (1-e^{-i\tau_1\lambda})^d(1-e^{i\tau_2\lambda})^d\frac{1}
{\lambda^{2d}} \langle f(\lambda)
\widetilde{e}_k,\widetilde{e}_k\rangle d\lambda
\\
=
\sum_{k=1}^\infty\langle R_{\xi^{(d)}}
\widetilde{e}_k,\widetilde{e}_k\rangle=\|{\zeta}_j\|^2_H\leq P_{\xi^{(d)}}.
\end{multline}

The stationary $d$th increment sequence $\vec{\xi}^{(d)}_j$ admits the spectral representation \cite{Karhunen}
\begin{equation}
\label{SpectrPred_incr_xi_c}
\vec{\xi}^{\,(d)}_j=\int_{-\pi}^{\pi}e^{i \lambda j}(1-e^{-i\tau\lambda})^{d}\frac{1}{(i\lambda)^{d}}d\vec{Z}_{\xi^{(d)}}(\lambda),
\end{equation}
where $\vec{Z}_{\xi^{(d)}}(\lambda)=\{Z_{k}(\lambda)\}_{k=1}^{\infty}$ is a vector-valued random process with uncorrelated increments on $[-\pi,\pi)$.

In the space $ H=L_2(\Omega, \cal F, \mt P)$, consider    a closed linear subspace
\[
H(\vec{\xi}^{\,(d)})=\overline{span}\{\xi^{(d)}_{kj}: k=1,2,\dots;\, j\in \mathbb Z \}\]
 generated by the components
of the generated stationary increment sequence
$\vec{\xi}^{\,(d)}=\{\xi^{(d)}_{kj}=\xi^{(d)}_{k}(j,\tau),\,\tau>0\}$.
For $q\in \mathbb Z$, consider also a closed linear subspace
\[H^{q}(\vec{\xi}^{\,(d)})=\overline{span}\{\xi^{(d)}_{kj}: k=1,2,\dots;\, j\leq q \}.\]
Define a subspace
\[ S(\vec{\xi}^{\,(d)})=\bigcap_{q\in \mathbb{Z}} H^{q}(\vec{\xi}^{\,(d)})
\]
  of the Hilbert space $H(\vec{\xi}^{\,(d)})$. The  space $H(\vec{\xi}^{\,(d)})$ admits a decomposition
$ H(\vec{\xi}^{\,(d)})=S(\vec{\xi}^{\,(d)})\oplus R(\vec{\xi}^{\,(d)}) $
where $R(\vec{\xi}^{\,(d)})$ is the orthogonal complement of the subspace $S(\vec{\xi}^{\,(d)})$ in the space $H(\vec{\xi}^{\,(d)})$.
\begin{ozn}
A stationary (wide sense)    increment sequence $\vec\xi^{(d)}_{j}=\{\xi^{(d)}_{kj}\}_{k=1}^{\infty}$ is called regular if $H(\vec{\xi}^{\,(d)})=R(\vec{\xi}^{\,(d)})$,
and it is called singular if
$H(\vec{\xi}^{\,(d)})=S(\vec{\xi}^{\,(d)})$.
\end{ozn}

\begin{thm}
A stationary    increment sequence $\xi^{(d)}_{j}$ is uniquely represented in the form
\begin{equation} \label{rozklad_cont}
\xi^{(d)}_{kj}
=\xi^{(d)}_{S,kj}+\xi^{(d)}_{R,kj}
\end{equation}
where  $\xi^{(d)}_{R,kj}, k=1,\ldots,\infty$, is a regular stationary     increment sequence and
$\xi^{(d)}_{S,kj}, k=1,\dots,\infty$, is a singular stationary   increment sequence.
The   increment sequences
$\xi^{(d)}_{R,kj}$  and
$\xi^{(d)}_{S,kj}$
are
orthogonal for all $ j\in\mathbb{Z} $. They are defined by the formulas
\begin{align*} \xi^{(d)}_{S,kj}&=\mt E[\xi^{(d)}_{kj}|S(\vec{\xi}^{\,(d)})],
\\
\xi^{(d)}_{R,kj}&=\xi^{(d)}_{kj}
-\xi^{(d)}_{S,kj}.
\end{align*}
\end{thm}

Consider an innovation sequence  ${\vec\varepsilon(u)=\{\varepsilon_m(u)\}_{m=1}^M, u \in\mathbb Z}$ for a regular stationary
increment, namely, the sequence of uncorrelated random
variables such that $\mathsf{E} \varepsilon_m(u)\overline{\varepsilon}_j(v)=\delta_{mj}\delta_{uv}$,   $\mathsf{E} |\varepsilon_m(u)|^2=1, m,j=1,\dots,M; u \in\mathbb Z$, and $H^{r}(\vec\xi^{(d)} )=H^{r}(\vec\varepsilon)$ holds true for all $r \in \mathbb Z$, where  $H^r(\vec\varepsilon)$ is
the Hilbert space generated by elements $ \{ \varepsilon_m(u):m=1,\dots,M; u\leq r\}$,
 $\delta_{mj}$ and $\delta_{uv}$ are Kronecker symbols.

\begin{thm}\label{thm 4_cont}
A   stationary   increment sequence
$\vec\xi^{(d)}_{j}$ is regular if and only if there exists an
innovation sequence ${\vec\varepsilon(u)=\{\varepsilon_m(u)\}_{m=1}^M, u \in\mathbb Z}$
and a sequence of matrix-valued
functions $\phi^{(d)}(l,\tau) =\{\phi^{(d)}_{km}(l,\tau) \}_{k=\overline{1,\infty}}^{m=\overline{1,M}}$, $l\geq0$, such that
\begin{equation}\label{odnostRuhSer_cont}
\sum_{l=0}^{\infty}
\sum_{k=1}^{\infty}
\sum_{m=1}^{M}
|\phi^{(d)}_{km}(l,\tau)|^2
<\infty,\quad
\xi^{(d)}_{kj}=
\sum_{l=0}^{\infty}\sum_{m=0}^{M}\phi^{(d)}_{km}(l,\tau)\vec\varepsilon_m(j-l).
\end{equation}
Representation (\ref{odnostRuhSer_cont}) is called the canonical
moving average representation of the generated stationary   increment
sequence $\vec\xi^{\,(d)}_{j}$.
\end{thm}

The spectral function $F(\lambda)$ of  a stationary    increment sequence $\vec\xi^{\,(d)}_{j}$
which admits   canonical representation
 $(\ref{odnostRuhSer_cont})$  has the spectral density  $f(\lambda)=\{f_{ij}(\lambda)\}_{i,j=1}^\infty$ admitting the canonical
factorization
\begin{equation}\label{SpectrRozclad_f_cont}
f(\lambda)=
\Phi(e^{-i\lambda})\Phi^*(e^{-i\lambda}),
\end{equation}
where the function
$\Phi(z)=\sum_{k=0}^{\infty}\phi(k)z^k$ has
analytic in the unit circle $\{z:|z|\leq1\}$
components
$\Phi_{ij}(z)=\sum_{k=0}^{\infty}\phi_{ij}(k)z^k; i=1,\dots,\infty; j=1,\dots,M$. Based on moving average  representation
$(\ref{odnostRuhSer_cont})$ define
\[\Phi_{\tau}(z)=\sum_{k=0}^{\infty} \phi^{(d)}(k,\tau)z^k=\sum_{k=0}^{\infty}\phi_{\tau}(k)z^k.\]
 Then the following factorization holds true:
\be
 \frac{|1-e^{i\lambda\tau}|^{2d}}{\lambda^{2d}}f(\lambda)= \Phi_{\tau}(e^{-i\lambda})
\Phi^*_{\tau}(e^{-i\lambda}),\quad \Phi_{\tau}(e^{-i\lambda})=\sum_{k=0}^{\infty}\phi_{\tau}(k)e^{-i\lambda k}.
 \label{dd_cont}
\ee

\section{Hilbert space projection method of filtering}\label{classical_filtering_f_n-n_c}

Let the uncorrelated increment processes $\{\xi^{(d)}(t,\tau T):t\in\mathbb R\}$ and $\{\eta^{(d)}(t,\tau T):t\in\mathbb R\}$ generate
 infinite dimensional vector stationary increment sequences $\{\vec{\xi}^{\,(d)}_j,j\in\mr Z\}$ and $\{\vec{\eta}^{\,(d)}_j,j\in\mr Z\}$ (\ref{generted_incr_seq}) with
 the spectral density matrices $f(\lambda)=\{f_{ij}(\lambda)\}_{i,j=1}^{\infty}$ and $g(\lambda)=\{g_{ij}(\lambda)\}_{i,j=1}^{\infty}$.

 Assume that the mean
values of the   sequences $\vec{\xi}^{\,(d)}_j$ and
  $\vec{\eta}^{\,(d)}_j$ equal to 0, and consider the increment step $\tau>0$.

By  the classical \textbf{filtering problem} we understand the problem of the mean square optimal linear  estimation  of the functional
\[ A{\xi}=\int_{0}^{\infty}a(t)\xi(-t)dt\]
which depends on the unknown values of the stochastic process $\xi(t)$ with periodically correlated $d$th increments.
Estimates are based on observations of the process $\zeta(t)=\xi(t)+\eta(t)$  at points $t\leq 0$.

To apply the Hilbert pace projection method of estimation, we represent the functional $A\xi$ in terms of the generated increment sequence $\vec{\xi}^{\,(d)}_j$.
\begin{lema}\label{lema predst A_f_n-n_c}
Any linear functional
\[
    A\xi=\int_{0}^{\infty}a(t)\xi(-t)dt\]
allows the representation
\[
    A\xi=B\xi,\]
where
\[
    B\xi=\int_{0}^{\infty}b^{\tau}(t)\xi^{(d)}(-t,\tau T)dt,\]
and
\begin{align}
    \label{koef b_cont}b^{\tau}(t)=\sum_{l=0}^{\ld[\frac{t}{\tau T}\rd]}a(t-\tau T l)d(l)=D^{\tau T}\me a(t) ,\,\,t\geq0,
\end{align}
$[x]$ denotes the  integer part of  $x$,  coefficients
  $\{d(l):l\geq0\}$ are determined by the relation
\[\sum_{l=0}^{\infty}d (l)x^l=\bigg(\sum_{j=0}^{\infty}x^{
j}\bigg)^d,\]
$D^{\tau T}$ is the linear transformation acting on an arbitrary function $x(t)$, $t\geq0$, as follows:
\[
    D^{\tau T}\me x(t)=\sum_{l=0}^{\ld[\frac{t}{\tau T}\rd]}x(t-\tau T l )d(l).\]
\end{lema}
\begin{proof}
From the definition of the increment process  obtain the formal equality
\be\label{obernene oznach prir}
 \xi(-t)=\frac{1}{(1-B_{\tau T})^d}\xi^{(d)}(-t,\tau T)
 =\sum_{j=0}^{\infty}d(j)\xi^{(d)}(-t-\tau Tj,\tau T).\ee
Relation (\ref{obernene oznach prir}) implies that
\[
 \int_0^{\infty}a(t)\xi(-t)dt=
 \int_0^{\infty}\xi^{(d)}(-t,\tau T)\sum_{j=0}^{\ld[\frac{t}{\tau T}\rd]}a(t-\tau Tj)d(j)dt\]
 which proves the statement of the lemma.
 \end{proof}

The functional $A\xi$ allows a representation  in terms of the increments sequence   $\vec\xi^{\,(d)}_j=(\xi^{(d)}_{kj},k=1,2,\dots)^{\top}$, $j\in\mr Z$, which is described in the following lemma.
\begin{lema}
   \label{lema predst A_f_n-n_c}
The linear functional $A\xi$ allows a representation
\[
 A\xi=B\xi=
 \sum_{j=0}^{\infty}
{(\vec{b}^{\tau}_j)}^{\top}\vec{\xi}^{\,(d)}_{-j}=:B\vec\xi,\]
where the infinite dimension vectors $\vec{\xi}^{\,(d)}_j$ and $\vec{b}^{\tau}_j$ are defined as follows:
\begin{align*}
\vec{\xi}^{\,(d)}_j&=(\xi^{(d)}_{kj},k=1,2,\dots)^{\top},
\\
\vec{b}^{\tau}_j&=(b^{\tau}_{kj},k=1,2,\dots)^{\top}=
(b^{\tau}_{1j},b^{\tau}_{3j},b^{\tau}_{2j},\dots,b^{\tau}_{2k+1,j},b^{\tau}_{2k,j},\dots)^{\top}.
\end{align*}
Here
\begin{align*}
 b^{\tau}_{kj}&=\langle b^{\tau}_j,\widetilde{e}_k\rangle =
\frac{1}{\sqrt{T}} \int_0^T b^{\tau}_j(v)e^{-2\pi
i\{(-1)^k\left[\frac{k}{2}\right]\}v/T}dv,
\\
   b^{\tau}_j(u)&=\sum_{l=0}^{j}a(u+jT-\tau T l)d_{\tau}(l)=D^{\tau} a_j(u),\, u\in [0,T), \, j=0,1,\ldots,
\end{align*}
and
\begin{align*}
    b^{\tau}_{k,j}&=\sum_{l=0}^{j}a_{k,j-\tau l}d_{\tau}(l),\, j=0,1,\ldots,\, k=1,2\ldots,
\\
 a_{kj}&=\langle a_j,\widetilde{e}_k\rangle =
\frac{1}{\sqrt{T}} \int_0^T a_j(v)e^{-2\pi
i\{(-1)^k\left[\frac{k}{2}\right]\}v/T}dv.
\end{align*}
\be \label{koef b_A_f_n-n_c}
\vec{b}^{\tau}_j=\sum_{m=0}^{j}\mt{diag}_{\infty}(d_{\tau}(j-m))\vec{a}_m=(D^{\tau}{\me a})_j,  \, j=0,1,2,\dots,
\ee
where
\[
\me a=((\vec{a}_0)^{\top},(\vec{a}_1)^{\top},(\vec{a}_2)^{\top}, \ldots)^{\top},
\]
\[
\vec{a}_j=(a_{kj},k=1,2,\dots)^{\top}=
(a_{1j},a_{3j},a_{2j},\dots,a_{2k+1,j},a_{2k,j},\dots)^{\top},
\]
the coefficients $\{d_{\tau}(k):k\geq0\}$ are determined by the relationship
\[
 \sum_{k=0}^{\infty}d_{\tau}(k)x^k=\bigg(\sum_{j=0}^{\infty}x^{\tau j}\bigg)^d,\]
$D^{\tau}$ is a linear transformation  determined by a matrix with infinite dimension matrix entries
 $D^{\tau}(k,j), k,j=0,1,2,\dots$ such that $D^{\tau}(k,j)=\mt{diag}_{\infty}(d_{\tau}(k-j))$ if $0\leq j\leq k $ and $D^{\tau}(k,j)=\mt{diag}_{\infty}(0)$ for $0\leq k<j$; $\mt{diag}_{\infty}(x)$ denotes an infinite dimension diagonal matrix with the entry $x$ on its diagonal.
 \end{lema}
\begin{proof}
The representation of the functional $A\vec{\xi}$ comes from the corresponding lemma in \cite{Luz_Mokl_extra_cont_PCI}.
\end{proof}

Suppose that the vectors $\vec{b}^{\tau}_j$, $j\geq0$, generated by the function $a(t)$, $t\geq0$, satisfy the  conditions
\begin{equation} \label{coeff_b_f_n-n_c}
 \sum_{j=0}^\infty\|
 \vec{b}^{\tau}_j\|<\infty, \quad
 \sum_{j=0}^\infty(j+1)\|
 \vec{b}^{\tau}_j\|<\infty, \quad
  \|\vec{b}^{\tau}_j\|^2=\sum_{k=1}^\infty |{b}^{\tau}_{kj}|^2.
\end{equation}
 Under such conditions  the functional $B\vec{\xi}$ belongs to the Hilbert space $H=L_2(\Omega,\mathcal{F}, P)$.

 \begin{remark}
  Restrictions (\ref{coeff_b_f_n-n_c}) on coefficients $\vec b^{\tau}_j$, $j\geq0$, are   strong. Let us show that they are not satisfied if $\vec b^{\tau}_j=\vec 0$ for $j\geq j_0$. Consider the simple case when $\|\vec b^{\tau}_0\|^2=\|\vec a_0\|^2=a$ and $\vec a_j=\vec 0$ for $j\geq 1$. Then the vectors $\vec{b}_j=\mt{diag}_{\infty}(d_{\tau}(j))\vec{a}_0$, $k\geq0$ and conditions (\ref{coeff_b_f_n-n_c}) are not satisfied.
  \end{remark}
   \begin{remark}
   An example when conditions (\ref{coeff_b_f_n-n_c}) are satisfied for the functional
   $$A_{NT}\xi=\int_{0}^{(N+1)T}a(t)\xi(-k)dt,$$
    is a class of functionals allowing representation as a functional
\[
 A_{NT}\xi=B_M\vec\xi=\sum_{j=0}^M(\vec b^{\tau}_j)^{\top}\vec\xi^{\,(d)}_{-j}\]
 which depends on the increments $\vec\xi^{\,(d)}_{-j}$ generated by the process, namely, $A_{NT}\xi=B_M\vec\xi$. In terms of the vectors $\vec a_j$, we have the following condition: there exist vectors $\{\vec b^{\tau}_j:j=0,1,2,\ldots,M\}$ such that
\[
 \vec a_j=\sum_{l=\max\ld\{0;\ld\lceil\frac{j-M}{\tau}\rd\rceil\rd\}}^{\min\ld\{d;\ld[\frac{j}{\tau}\rd]\rd\}}
 (-1)^l{n \choose l}\vec b^{\tau}_{j-\tau l}\]
 for $j=0,1,2,\ldots,N$, $N=M+\tau d$. By $\lceil x\rceil$ we denote the least integer greater than or equal to  $x$.
\end{remark}

Denote by
$\Delta(f,g;\widehat{A}{\xi})=\mt E |A{\xi}-\widehat{A}{\xi}|^2$ and
$\Delta(f,g;\widehat{B}\vec\xi):=\mt E |B\vec\xi-\widehat{B}\vec\xi|^2$
 the mean square errors (MSE) of the optimal estimates $\widehat{A}\xi$ and $\widehat{B}\vec\xi$ of the functionals $A\xi$ and $B\vec\xi$.
 Since $A\xi=B\vec\xi$, we have the following relations:
\be\label{main
formula}
 \widehat{A}\xi=\widehat{B}\vec\xi,\ee
and
\[
 \Delta(f,g;\widehat{A}\xi)=\mt E |A\xi-\widehat{A}\xi|^2=
 \mt E|B\vec\xi-\widehat{B}\vec\xi|^2=\Delta(f,g;\widehat{B}\vec\xi).\]

Let the spectral densities $f(\lambda)$ and $g(\lambda)$ satisfy the minimality condition
\be
 \ip \dfrac{\lambda^{2d}}{|1-e^{i\lambda\tau}|^{2d}}(f(\lambda)+g(\lambda))^{-1}
 d\lambda<\infty.\label{umova11_f_n-n_c}\ee
The estimate of the functional $ B\vec\xi$ can be found applying the method of orthogonal projection in the Hilbert space.
Define a closed linear subspace
\[H^{0}(\vec\xi^{\,(d)}+\vec\eta^{\,(d)})
=\overline{span}\{\xi^{(d)}_k(j,\tau)+\eta^{(d)}_k(j,\tau):k=1,\dots,\infty;\,j\leq 0\}\]
of $H=L_2(\Omega,\mathcal{F},\mt P)$ generated by the observations $\{\zeta(t):t\leq0\}$, and
 a closed linear subspace
\[
L_2^{0}(f+g)=\overline{span}\{
 e^{i\lambda j}(1-e^{-i\lambda \tau})^d\frac{1}{(i\lambda)^d}\vec\delta_{k}:
 \,
  k=1,\dots,\infty;\,j\leq 0\}
 \]
of the Hilbert space
$L_2(f+g)$  of vector-valued functions endowed by the inner product $\langle g_1;g_2\rangle=\ip (g_1(\lambda))^{\top}(f(\lambda)+g(\lambda))\overline{g_2(\lambda)}d\lambda$. Here   $\vec\delta_k=\{\delta_{kl}\}_{l=1}^{\infty}$, $\delta_{kl}$ are Kronecker symbols. Note, that $f(\lambda)+g(\lambda)$   is a spectral density matrix of the stationary increment sequence $\vec{\xi}^{\,(d)}_j+\vec{\eta}^{\,(d)}_j$.

\begin{zau}
Representation
 (\ref{SpectrPred_incr_xi_c})
yields a  map between a function
\[
e^{i\lambda j}(1-e^{-i\lambda\tau})^d(i\lambda)^{-d}\vec\delta_k\]
 from the space
$L_2^{0}(f+g)$
and an increment
$\xi^{(d)}_k(j,\tau)+\eta^{(d)}_k(j,\tau)$
from the space
$H^{0}(\vec\xi^{\,(d)}+\vec\eta^{\,(d)})$.
\end{zau}

 The linear estimates $\widehat{A}\xi$ and $\widehat{B} \vec\xi$ have the following representations:
\be \label{otsinka A_f_n-n_c}
 \widehat{A}\xi=\widehat{B}\vec \xi=\ip
 (\vec h_{\tau}(\lambda))^{\top}d\vec Z_{\xi^{(d)}+\eta^{(d)}}(\lambda), \ee
where $\vec h_{\tau}(\lambda)$ is a spectral characteristic
of the optimal estimate $\widehat{B}\vec \xi$.  The
 estimate $\widehat{B} \xi$ is a projection of the element $B\xi$ of the Hilbert space $H=L_2(\Omega,\mathcal{F}, P)$ on the
subspace $H^{0}(\vec\xi^{\,(d)}+\vec\eta^{\,(d)})$:
\[
\widehat{A}\xi=\widehat{B}\vec\xi=\mt {Proj}_{H^{0}(\vec\xi^{\,(d)}+\vec\eta^{\,(d)})}B\vec\xi.
\]

Define the following Fourier coefficients for all $l,j\geq0$:
\begin{align*}
  R_{l,j}&=\frac{1}{2\pi}\ip
 e^{-i\lambda(l+j)}(f(\lambda)(f(\lambda)+g(\lambda))^{-1})^{\top}d\lambda,
 \\
  P_{l,j}^{\tau}&=\frac{1}{2\pi}\ip e^{-i\lambda (l-j)}
 \dfrac{\lambda^{2d}}{|1-e^{i\lambda\tau}|^{2d}}\ld((f(\lambda)+g(\lambda)\rd)^{-1})^{\top}
 d\lambda,
 \\
  Q_{l,j}^{\tau}&=\frac{1}{2\pi}\ip
 e^{-i\lambda (l-j)}\frac{|1-e^{i\lambda\tau}|^{2d}}
 {\lambda^{2d}}(f(\lambda)(f(\lambda)+g(\lambda))^{-1}g(\lambda))^{\top}d\lambda.
  \end{align*}
Define a vector
\[
  \me b^{\tau}  =((\vec b^{\tau}_0)^{\top},(\vec b^{\tau}_1)^{\top},
  (\vec b^{\tau}_2)^{\top},\ldots)^{\top},
\]
and
linear operators $\me P_{\tau}$, $\me R$ and $\me Q_{\tau}$ in the space $\ell_2$ determined by matrices with the infinite dimensional entries
$(\me P_{\tau})_{l,j}=P_{l,j}^{\tau}$,   $(\me
R)_{l,j} =R_{l+1,j}$, $(\me Q_{\tau})_{l,j}=Q_{l,j}^{\tau}$, $l,j\geq0$.

The following theorem provides a solution of the classical filtering problem.

\begin{thm}
  \label{thm1_f_n-n_c}
Consider two uncorrelated
   stochastic processes $\xi(t)$ and $\eta(t)$, $t\in \mr R$, with  periodically stationary  increments, which determine
generated stationary  $d$th increment sequences
$\vec{\xi}^{\,(d)}_j$ and $\vec{\eta}^{\,(d)}_j$ with the spectral density matrices $f(\lambda)=\{f_{kn}(\lambda)\}_{k,n=1}^{\infty}$ and $g(\lambda)=\{g_{kn}(\lambda)\}_{k,n=1}^{\infty}$, respectively.
Let the coefficients $\vec {b}^{\tau}_j$, $j=0,1,\ldots,$ generated by the function $a(t)$, $t\geq 0$, satisfy conditions  (\ref{coeff_b_f_n-n_c}).
 Let    minimality condition
(\ref{umova11_f_n-n_c}) be satisfied.
The optimal linear estimate
$\widehat{A}\xi$ of the functional $A\xi$ which depends on the unknown
values $\xi(t)$  for $t\leq0$  based
on observations of the   process $\xi(t)+\eta(t)$
at points $t\leq 0$ is determined by formula (\ref{otsinka A_f_n-n_c}).
The spectral characteristic $\vec h_{\tau}(\lambda)$ of the optimal estimate
 is calculated by the formula
 \begin{align}
 \notag (\vec h_{\tau}(\lambda))^{\top}&=
 (\vec B_{\tau}(e^{-i\lambda}))^{\top}\frac{(1-e^{-i\lambda\tau})^d}{(i\lambda)^{d}}
 f(\lambda)(f(\lambda)+g(\lambda))^{-1}
 \\
 &\quad-\frac{(-i\lambda)^{d}}{(1-e^{i\lambda\tau})^d}(\vec C_{\tau}(e^{i\lambda}))^{\top}(f(\lambda)+g(\lambda))^{-1},
 \label{spectr A_f_n-n_c}
\end{align}
where
 \begin{align*}
 \vec B_{\tau}(e^{-i\lambda })&=\sum_{j=0}^{\infty}\vec b^{\tau}_je^{-i\lambda j}
 =\sum_{j=0}^{\infty}( D^{\tau} \me a)_je^{-i\lambda j},
\\
 \vec C_{\tau}(e^{i\lambda})&=\sum_{j=0}^{\infty}
 \ld(\me P_{\tau}^{-1}\me RD^{\tau}\me a\rd)_j e^{i\lambda
 (j+1)}.
 \end{align*}
 The MSE of the estimate $\widehat{A}\xi$
is calculated by the formula
\begin{align}
 \notag &\Delta(f,g;\widehat{A}\xi)=\Delta(f,g;\widehat{B}\vec\xi)= \mt E|B\vec\xi-\widehat{B}\vec\xi|^2=
 \\\notag
 &=\frac{1}{2\pi}\ip
 \frac{\lambda^{2d}}
 {|1-e^{i\lambda\tau}|^{2d}}
 \bigg((\vec B_{\tau}(e^{-i\lambda}))^{\top}\frac{|1-e^{i\lambda
 \tau}|^{2d}}{\lambda^{2d}}g(\lambda)+(\vec C_{\tau}(e^{i\lambda}))^{\top}\bigg)
 \\
\notag &\quad\quad\times
 (f(\lambda)+g(\lambda))^{-1} f(\lambda)(f(\lambda)+g(\lambda))^{-1}
 \\
\notag &\quad\quad\times
\bigg(g(\lambda)\overline{\vec B_{\tau}(e^{-i\lambda})}\frac{|1-e^{i\lambda
 \tau}|^{2d}}{\lambda^{2d}}+\overline{\vec C_{\tau}(e^{i\lambda})}\bigg)
  d\lambda
 \\
 \notag
 &\quad+\frac{1}{2\pi}\ip
 \frac{\lambda^{2d}}
 {|1-e^{i\lambda\tau}|^{2d}}
 \bigg((\vec B_{\tau}(e^{-i\lambda}))^{\top}\frac{|1-e^{i\lambda
 \tau}|^{2d}}{\lambda^{2d}}f(\lambda)--(\vec C_{\tau}(e^{i\lambda}))^{\top}\bigg)
 \\
\notag &\quad\quad\times
 (f(\lambda)+g(\lambda))^{-1}g(\lambda)(f(\lambda)+g(\lambda))^{-1}
 \\
\notag &\quad\quad\times
\bigg(f(\lambda)\overline{\vec B_{\tau}(e^{-i\lambda})}\frac{|1-e^{i\lambda
 \tau}|^{2d}}{\lambda^{2d}}-\overline{\vec C_{\tau}(e^{i\lambda})}\bigg)
  d\lambda
 \\
 &=\ld\langle \me R
 D^{\tau}\me a,\me P_{\tau}^{-1}\me R D^{\tau}\me a\rd\rangle+\ld\langle\me Q_{\tau} D^{\tau}\me a,
 D^{\tau}\me a\rd\rangle.\label{poh A_f_n-n_c}\end{align}
\end{thm}
  \begin{proof}
See Appendix.
\end{proof}

\section{Filtering based on factorizations of the spectral densities}
\label{classical_filtering_fact_f_n-n_c}

In this section, we  derive the   formulas for the spectral characteristic and the MSE of the estimate $\widehat{A}\xi$  in terms of coefficients of factorizations of the    spectral densities $f(\lambda)$ and  $g(\lambda)$.
 Suppose that   the following canonical factorization take place
\begin{multline}
 \frac{|1-e^{i\lambda\tau}|^{2d}}{\lambda^{2d}}
 (f(\lambda)+g(\lambda))
 =\Theta_{\tau}(e^{-i\lambda})\Theta_{\tau}^*(e^{-i\lambda}),
 \\
   \Theta_{\tau}(e^{-i\lambda})=\sum_{k=0}^{\infty}\theta_{\tau}(k)
  e^{-i\lambda k}, \label{fakt1_f_n-n_c}
\end{multline}
Define the matrix-valued function
$\Psi_{\tau}(e^{-i\lambda})= \{\Psi_{ij}(e^{-i\lambda})\}_{i=\overline{1,M}}^{j=\overline{1,\infty}}$
by the equation
\[
\Psi_{\tau}(e^{-i\lambda})\Theta_{\tau}(e^{-i\lambda})=E_M,
\]
where $E_M$ is an identity $M\times M$ matrix.
Then the following factorization takes place
\begin{multline}
\frac{\lambda^{2d}}{|1-e^{i\lambda\tau}|^{2d}}
 (f(\lambda)+g(\lambda))^{-1} =
 \Psi_{\tau}^*(e^{-i\lambda})\Psi_{\tau}(e^{-i\lambda}),
 \\
 \Psi_{\tau}(e^{-i\lambda})=\sum_{k=0}^{\infty}\psi_{\tau}(k)e^{-i\lambda k}.\label{fakt2_f_n-n_c}
\end{multline}
 Define the linear operators
 $\Psi_{\tau}$, $ \Theta_{\tau}$ in the space $\ell_2$ by the matrices with the matrix entries
 $(\Psi_{\tau})_{k,j}=\psi_{\tau}(k-j)$, $(\Theta_{\tau})_{k,j}=\theta_{\tau}(k-j)$ for $0\leq j\leq k$ and $(
\Psi_{\tau})_{k,j}=0$, $(\Theta_{\tau})_{k,j}=0$ for $0\leq k<j$.

\begin{lema}\label{lema_fact_2_e_n_c}
Let factorization (\ref{fakt1_f_n-n_c})  take  place and let $M\times \infty$ matrix function $\Psi_{\tau}(e^{-i\lambda})$ satisfy equation $\Psi_{\tau}(e^{-i\lambda})\Theta_{\tau}(e^{-i\lambda})=E_M$. Define the linear operators
 $ \Psi_{\tau}$ and  $ \Theta_{\tau}$ in the space $\ell_2$ by the matrices with the matrix entries
 $( \Psi_{\tau})_{k,j}=\psi_{\tau}(k-j)$, $( \Theta_{\tau})_{k,j}=\theta_{\tau}(k-j)$ for $0\leq j\leq k$, $(
\Psi_{\tau})_{k,j}=0$, $(
\Theta_{\tau})_{k,j}=0$ for $0\leq k<j$.
Then:
\\
a) the linear operator $\me P_{\tau}$  admits the factorization \[\me
P_{\tau}=(\Psi_{\tau})^{\top} \overline{\Psi}_{\tau};\]
\\
b) the inverse operator $(\me
P_{\tau})^{-1}$ admits the factorization
\[
 (\me
P_{\tau})^{-1}= \overline{\Theta}_{\tau}(\Theta_{\tau})^{\top}.\]
\end{lema}

\begin{remark}
The entries of the matrix of the operator $(\me P_{\tau})^{-1}$
 are calculated by the formula
\[
 ((\me P_{\tau})^{-1})_{l,j}=\sum_{p=0}^{\min(l,j)}\overline{\theta}_{\tau}(j-p)(\theta_{\tau}(l-p))^{\top},\quad l,j\geq0.\]
\end{remark}

Let the canonical factorizations
\begin{multline} \label{fakt3_f_n-n_c}
 \frac{|1-e^{i\lambda\tau}|^{2d}}{\lambda^{2d}} f(\lambda) =\sum_{k=-\infty}^{\infty}f_{\tau}(k)e^{i\lambda k}=
 \Phi_{\tau}(e^{-i\lambda})\Phi_{\tau}^*(e^{-i\lambda}),
 \\
  \Phi_{\tau}(e^{-i\lambda})=\sum_{k=0}^{\infty}\phi_{\tau}(k)e^{-i\lambda k}
 \end{multline}
 and
 \begin{multline}
 \frac{|1-e^{i\lambda\tau}|^{2d}}{\lambda^{2d}} g(\lambda) =\sum_{k=-\infty}^{\infty}g_{\tau}(k)e^{i\lambda k}=
 \Omega_{\tau}(e^{-i\lambda})\Omega_{\tau}^*(e^{-i\lambda}),
 \\
  \Omega_{\tau}(e^{-i\lambda})=\sum_{k=0}^{\infty}\omega_{\tau}(k)e^{-i\lambda k}
  \label{fakt6_f_n-n_c}
 \end{multline}
 take place.

\begin{lema}\label{lema_fact_3_f_n-c_c}
  Let    factorizations (\ref{fakt1_f_n-n_c}), (\ref{fakt3_f_n-n_c}) and (\ref{fakt6_f_n-n_c}) take place.
 Define by $\vec e^{\,\tau}_m=(\Theta^{\top}_{\tau}\me R\me b^{\tau})_m$, $m\geq0$, the $m$th element of the vector
$\widetilde{\me e}_{\tau}=\Theta^{\top}_{\tau}\me R \me b^{\tau}$. Then
\[
\vec e^{\,\tau}_m=\sum_{j=0}^{\infty}Z_{\tau}(m+j+1)\vec b^{\tau}_j=-
 \sum_{j=0}^{\infty}X_{\tau}(m+j+1)\vec b^{\tau}_j,
\]
  where $Z_{\tau}(j)$ and $X_{\tau}(j)$  are defined as
\begin{align*}
  Z_{\tau}(j)&=\sum_{l=0}^{\infty}\overline{\psi}_{\tau}(l)\overline{f}_{\tau}(l-j),\, j\in \mr Z,
 \\
 f_{\tau}(k)&=\sum_{m=\max\{0,-k\}}^{\infty}\phi_{\tau}(m)\phi_{\tau}^*(k+m),\, k\in\mr Z,
\end{align*}
and
\begin{align*}
  X_{\tau}(j)&=\sum_{l=0}^{\infty}\overline{\psi}_{\tau}(l)\overline{g}_{\tau}(l-j),\, j\in \mr Z,
 \\
 g_{\tau}(j)&=\sum_{m=\max\{0,-j\}}^{\infty}\omega_{\tau}(m)\omega_{\tau}^*(j+m),\, j\in\mr Z.
\end{align*}
\end{lema}
\begin{proof}
See Appendix.
\end{proof}

Define  the linear operators $\me G^{f}_{\tau}$ and $\me G^{g}_{\tau}$  in the space $\ell_2$ by matrices with the matrix entries $(\me G^{f}_{\tau})_{l,j}=\overline{f}_{\tau}(l-j)$ and $(\me G^{g}_{\tau})_{l,j}=\overline{g}_{\tau}(l-j)$ $l,j\geq0$.  And the linear operators
  $\widetilde{\Phi}_{\tau}$ and $\widetilde{\Omega}_{\tau}$ in the space $\ell_2$ determined by a matrix with the matrix entries
  $(\widetilde{\Phi}_{\tau})_{l,j}=\phi^{\top}_{\tau}(l-j)$ and  $(\widetilde{\Omega}_{\tau})_{l,j}=\omega^{\top}_{\tau}(l-j)$ for $0\leq j\leq l$,  $ (\widetilde{\Phi}_{\tau})_{l,j}=0$ and $ (\widetilde{\Omega}_{\tau})_{l,j}=0$ for $0\leq l<j$.

\begin{thm}\label{thm3_f_n-n_c_fact}
Consider two   stochastic processes with
periodically stationary increments, ${\xi}(t)$ and ${\eta}(t)$, generating  stationary increment sequences $\vec\xi^{\,(d)}_j$ and $\vec\eta^{\,(d)}_j$
with  spectral densities $f(\lambda)$ and $g(\lambda)$.      Assume that condition   (\ref{coeff_b_f_n-n_c}) are satisfied. The solution $\widehat{A}\xi$ to the filtering problem for the linear functional $A{\xi}$
is  calculated by formula
(\ref{otsinka A_f_n-n_c}).
If   factorizations (\ref{fakt1_f_n-n_c})--(\ref{fakt2_f_n-n_c}) take place, then  the spectral characteristic $\vec h_{\tau}(\lambda)$ of the estimate $\widehat{A}\xi$ can be calculated by the formula
\begin{align}
 \vec h_{\tau}(\lambda)=\frac{(1-e^{-i\lambda\tau})^d}{(i\lambda)^{d}}
\ld(\sum_{k=0}^{\infty}\psi^{\top}_{\tau}(k)e^{-i\lambda k}\rd)
\sum_{m=0}^{\infty}(\overline{\psi}_{\tau} \me C^{f}_{\tau})_m e^{-i\lambda m},\label{spectr A_f_n-n_c_fact}
\end{align}
where $\overline{\psi}_{\tau}=(\overline{\psi}_{\tau}(0),\overline{\psi}_{\tau}(1),\overline{\psi}_{\tau}(2),\ldots)$,
\[
(\overline{\psi}_{\tau} \me C^{f}_{\tau} )_m=\sum_{j=0}^{\infty}\overline{\psi}_{\tau}(j)\me c^{f}_{\tau}(j+m),
\]
\begin{align*}
  \me c^{f}_{\tau}(m)
  =\sum_{l=0}^{\infty}\overline{\phi}_{\tau}(l)(\widetilde{\Phi}_{\tau}D^{\tau}\me a)_{l+m}
  =(\widetilde{\Phi}_{\tau}^*\widetilde{\Phi}_{\tau}D^{\tau}\me a)_{m},
\end{align*}
and the  MSE is calculated by the formula
\be
 \Delta(f,g;\widehat{A}\xi)=\|\widetilde{\Phi}_{\tau}D^{\tau}\me a\|^2-\|\overline{\psi}_{\tau} \me C^{f}_{\tau}\|^2.\label{poh A_f_n_d_fact}
\ee
If factorizations (\ref{fakt1_f_n-n_c}) and (\ref{fakt6_f_n-n_c})
   take place, then
   the spectral characteristic $\vec h_{\tau}(\lambda)$ can be calculated by the formula
\begin{align}
 \vec h_{\tau}(\lambda)=\frac{(1-e^{-i\lambda\tau})^d}{(i\lambda)^{d}}\ld(\vec B_{\tau}(e^{-i\lambda})
-\ld(\sum_{k=0}^{\infty}\psi^{\top}_{\tau}(k)e^{-i\lambda k}\rd)
\sum_{m=0}^{\infty}(\overline{\psi}_{\tau} \me C^{g}_{\tau})_me^{-i\lambda m}\rd),
\label{spectr2 A_f_n-n_c_fact}
\end{align}
where
\[
(\overline{\psi}_{\tau} \me C^{g}_{\tau} )_m=\sum_{k=0}^{\infty}\overline{\psi}_{\tau}(k)\me c^{g}_{\tau}(k+m),
\]
\begin{align*}
  \me c^{g}_{\tau}(m)
  =\sum_{l=0}^{\infty}\overline{\omega}_{\tau}(l)(\widetilde{\Omega}_{\tau}D^{\tau}\me a)_{l+m}
  =(\widetilde{\Omega}_{\tau}^*\widetilde{\Omega}_{\tau}D^{\tau}\me a)_{m},
\end{align*}
and the  MSE is calculated by the formula
\be
 \Delta(f,g;\widehat{A}\xi)=\|\widetilde{\Omega}_{\tau}D^{\tau}\me a\|^2-\|\overline{\psi}_{\tau} \me C^{g}_{\tau}\|^2.\label{poh A2_f_n-n_d_fact}
\ee
\end{thm}
\begin{proof}
See Appendix.
\end{proof}
\begin{zau}
  \label{zauv_operator_U_f_n-n_c}
Let \cite{Luz_Mokl_book}
\[
 \frac{|1-e^{i\lambda\tau}|^{2d}}{\lambda^{2d}}=\ld|w_{\tau}(e^{-i\lambda})\rd|^2=\ld|\sum_{k=0}^{\infty}w_{\tau}(k)e^{-i\lambda k}\rd|^2.
 \]
Suppose, for example, that  factorization (\ref{fakt2_f_n-n_c}) takes place.
Then the spectral density $f(\lambda)+g(\lambda)$
admits a  factorization
\be
  f(\lambda)+g(\lambda)
 =\Theta(e^{-i\lambda})\Theta^*(e^{-i\lambda}),\quad
  \Theta(e^{-i\lambda})=\sum_{k=0}^{\infty}\theta(k)
  e^{-i\lambda k} \label{fakt5_f_n-n_c}
\ee
and the following relations hold true:
\[
 \Theta_{\tau}(e^{-i\lambda}) =w_{\tau}(e^{-i\lambda})\Theta (e^{-i\lambda}),
\]
\be
 \theta_{\tau}(k)= \sum_{m=0}^kw_{\tau}(k-m)\theta (m),\quad k=0,1,\dots
 \label{relation2_f_n-n_c}
 \ee
 or
 \[
 \theta_{\tau,ij}(k)= \sum_{m=0}^kw_{\tau}(k-m)\theta_{ij} (m), \quad i=1,\dots,\infty;\,j=1,\dots,M;\,k=0,1,\dots.
 \]
 Define the matrix-valued function
$\Psi(e^{-i\lambda})= \{\Psi_{ij}(e^{-i\lambda})\}_{i=\overline{1,M}}^{j=\overline{1,\infty}}$
by the equation
\[
\Psi(e^{-i\lambda})\Theta(e^{-i\lambda})=E_M,
\]
where $E_M$ is an identity $M\times M$ matrix. Then the following factorization takes place
\[
 (f(\lambda)+g(\lambda))^{-1} =
 \Psi^*(e^{-i\lambda})\Psi(e^{-i\lambda}), \quad \Psi(e^{-i\lambda})=\sum_{k=0}^{\infty}\psi(k)e^{-i\lambda k},\label{fakt8_f_n-n_c}
\]
 and
 \be
 \Psi(e^{-i\lambda})=\Psi_{\tau}(e^{-i\lambda})w_{\tau}(e^{-i\lambda}).
 \label{formula_11_f_n-n_c}
 \ee
 Define a linear operator $ \me W_{\tau}$ in the space $\ell_2$ by a matrix with the entries
 $(\me W_{\tau})_{k,j}=w_{\tau}(k-j)$ for $0\leq j\leq k$, $ (\me W_{\tau})_{k,j}=0$ for $0\leq k<j$.
 Relation (\ref{relation2_f_n-n_c}) can be represented in the form
 \begin{equation}
 \boldsymbol{\theta}_{\tau} =\me W^{\tau}\boldsymbol{\theta},
 \label{relation_for_psi_f_n-n_c}
 \end{equation}
where
\begin{align*}
  \boldsymbol{\theta}_{\tau}&=((\theta_{\tau}(0))^{\top},
(\theta_{\tau}(1))^{\top},(\theta_{\tau}(2))^{\top},\ldots)^{\top},
\\
 \boldsymbol{\theta} &=((\theta (0))^{\top},(\theta (1))^{\top},(\theta (2))^{\top},\ldots)^{\top}
\end{align*}
are vectors with the matrix entries
$\theta_{\tau}(k)=\{\theta_{\tau,ij}(k)\}_{i=\overline{1,\infty}}^{j=\overline{1,M}},\, k=0,1,2,\dots$, and
$\theta(k)=\{\theta_{ij}(k)\}_{i=\overline{1,\infty}}^{j=\overline{1,M}},\, k=0,1,2,\dots$.
\end{zau}

\section{Minimax filtering based on observations with periodically stationary increment noise}
\label{minimax_filtering}

\subsection{Least favorable spectral densities and minimax-robust spectral characteristic}

Consider the filtering problem for the functional ${A}\xi$ based on the observations
$\xi(t)+\eta(t)$ at points $t\leq0$ when the spectral densities of the generated stationary increments sequences are not exactly known
while a set $\md D=\md D_f\times\md D_g$ of admissible spectral densities is defined.
Here $\xi(t)$ and $\eta(t)$ are two uncorrelated processes with periodically stationary increments.
We apply the minimax (robust) approach of
estimation, which is formalized by the following two definitions.

\begin{ozn}
  For a given class of spectral densities $\mathcal{D}=\md
D_f\times\md D_g$ the spectral densities
$f^0(\lambda)\in\mathcal{D}_f$, $g^0(\lambda)\in\md D_g$
are called the least favourable spectral densities in the class $\mathcal{D}$ for the
optimal linear filtering of the functional $A\xi$ if
\[
 \Delta(f^0,g^0)=\Delta(h(f^0,g^0);f^0,g^0)=
 \max_{(f,g)\in\mathcal{D}_f\times\md D_g}\Delta(h(f,g);f,g).\]
\end{ozn}
\begin{ozn}
  For a given class of spectral
densities $\mathcal{D}=\md D_f\times\md D_g$ the spectral
characteristic $h^0(e^{i\lambda})$ of the optimal estimate of the functional
$A\xi$ is called minimax (robust) if
\[
 h^0(e^{i\lambda})\in H_{\mathcal{D}}=\bigcap_{(f,g)\in\mathcal{D}_f\times\md D_g}L_2^{0 }(f+g),\]
\[
 \min_{h\in H_{\mathcal{D}}}\max_{(f,g)\in \mathcal{D}_f\times\md D_g}\Delta(h;f,g)=\sup_{(f,g)\in\mathcal{D}_f\times\md
D_g}\Delta(h^0;f,g).\]
\end{ozn}

Taking into account the introduced definitions and the  relations derived in this chapter, as well as Remark \ref{zauv_operator_U_f_n-n_c} for the spectral densities
$f(\lambda)$, $g(\lambda)$, $f(\lambda)+g(\lambda)$, we can verify that the following lemmas hold true.

\begin{lema}
  The spectral densities $f^0\in\mathcal{D}_f$,
$g^0\in\mathcal{D}_g$ which satisfy condition
(\ref{umova11_f_n-n_c}) are the least favourable densities in the class $\mathcal{D}$ for the optimal
linear filtering of the functional $A\xi$ if operators $\me
P_{\tau}^0$, $\me R^0$, $\me Q_{\tau}^0$ defined by
the Fourier coefficients of the functions
\[
\dfrac{\lambda^{2d}}{|1-e^{i\lambda\tau}|^{2d}}
\ld((f^0(\lambda)+g^0(\lambda))^{-1}\rd)^{\top},\]
\[
\ld(
g^0(\lambda)(f^0(\lambda)+g^0(\lambda))^{-1}\rd)^{\top},\]
\[
\dfrac{|1-e^{i\lambda\tau}|^{2d}}{\lambda^{2d}}
\ld(f^0(\lambda)(f^0(\lambda)+g^0(\lambda))^{-1}g^0(\lambda)\rd)^{\top}
\]
determine a solution of the constrained optimization problem
 \begin{multline} \max_{f\in\mathcal{D}}\ld(\ld\langle
\me R
 D^{\tau}\me a,\me P_{\tau}^{-1}\me R D^{\tau}\me
a\rd\rangle+\langle\me Q_{\tau} D^{\tau}\me a,
 D^{\tau}\me a\rangle\rd)
 \\= \ld\langle \me R^0
 D^{\tau}\me a,(\me
P_{\tau}^0)^{-1}\me R^0 D^{\tau}\me a \rd\rangle+\langle\me
Q_{\tau}^0 D^{\tau}\me a, D^{\tau}\me a\rangle.
\label{minimax1_f_n-n_c}\end{multline}
\end{lema}

\begin{lema}
  The spectral densities $f^0\in\mathcal{D}_f$,
$g^0\in\mathcal{D}_g$ which admit the canonical factorizations (\ref{fakt1_f_n-n_c}), (\ref{fakt3_f_n-n_c}) or canonical factorizations (\ref{fakt1_f_n-n_c}), (\ref{fakt6_f_n-n_c})
are the least favourable densities in the class $\mathcal{D}$ for the optimal
linear filtering of the functional $A\xi$ based on observations of the process $\xi(t)+\eta(t)$ at points $t\leq0$ if matrix coefficients $\{\psi^0(k), \phi^0(k), \omega^0(k): k\geq0\}$
of the canonical factorizations
\be \label{fakt231_f_n-n_c}
 (f^0(\lambda)+g^0(\lambda))^{-1} =
 \ld(\sum_{k=0}^{\infty}\psi^0(k)e^{-i\lambda k}\rd)^*
 \ld(\sum_{k=0}^{\infty}\psi^0(k)e^{-i\lambda k}\rd),
 \ee
 \be \label{fakt232_f_n-n_c}
 f^0(\lambda)
 =\ld(\sum_{k=0}^{\infty}\phi^0(k)e^{-i\lambda k}\rd)
 \ld(\sum_{k=0}^{\infty}\phi^0(k)e^{-i\lambda k}\rd)^*,
 \ee
\be \label{fakt233_f_n-n_c}
 g^0(\lambda) =
 \ld(\sum_{k=0}^{\infty}\omega^0(k)e^{-i\lambda k}\rd)
 \ld(\sum_{k=0}^{\infty}\omega^0(k)e^{-i\lambda k}\rd)^*
\ee
determine a solution of the constrained optimization problem
\be\langle\widetilde{\Phi}_{\tau}D^{\tau}\me a^{\tau},\widetilde{\Phi}_{\tau}D^{\tau}\me a^{\tau}\rangle
 -\langle \overline{\psi}_{\tau}\me C_{\tau}^f ,\overline{\psi}_{\tau} \me C_{\tau}^f \rangle\rightarrow\sup, \label{minimax1f_f_n-n_c_fact}\ee
\begin{align*}
 f(\lambda)&=\Phi(e^{-i\lambda})\Phi^*(e^{-i\lambda})\in \mathcal{D}_f,
 \\
 g(\lambda)&=
\Theta(e^{-i\lambda})\Theta^*(e^{-i\lambda})
-\Phi(e^{-i\lambda})\Phi^*(e^{-i\lambda})\in \mathcal{D}_g,
\end{align*}
or of the constrained optimization problem
\be\langle \Omega_{\tau}D^{\tau}\me a_{\tau},\Omega_{\tau}D^{\tau}\me a_{\tau}\rangle
 -\langle\overline{\psi}_{\tau}\me C_{\tau}^g ,\overline{\psi}_{\tau} \me C_{\tau}^g \rangle\rightarrow\sup, \label{minimax1g_f_n-n_c_fact}\ee
\begin{align*}
 f(\lambda)&=\Theta(e^{-i\lambda})\Theta^*(e^{-i\lambda})
 -\Omega(e^{-i\lambda})\Omega^*(e^{-i\lambda})\in \mathcal{D}_f,
 \\
 g(\lambda)&=\Omega(e^{-i\lambda})\Omega^*(e^{-i\lambda})\in \mathcal{D}_g,
 \end{align*}
where  $\Psi(e^{-i\lambda})\Theta(e^{-i\lambda})=E_M$. The minimax spectral characteristic $h^0=\vec h_{\tau}(f^0,g^0)$ is calculated by formula (\ref{spectr A_f_n-n_c_fact}) or (\ref{spectr2 A_f_n-n_c_fact}) if
$\vec h_{\tau}(f^0,g^0)\in H_{\mathcal{D}}$.
\end{lema}

\begin{lema}The spectral density $f^0\in\mathcal{D}_f$ which admits the canonical factorizations (\ref{fakt1_f_n-n_c}), (\ref{fakt3_f_n-n_c}) with a known regular spectral density $g(\lambda)$ is the least favourable spectral density in the class
 $\md D_f$ for the optimal
linear filtering of the functional $A\xi$ based on observations of the process $\xi(t)+\eta(t)$ at points $t\leq0$ if the matrix coefficients $\{\psi^0(k): k\geq0\}$
from the canonical factorization
\be \label{fakt2_lf_f_n-n_c}
 (f^0(\lambda)+g(\lambda))^{-1} =
 \ld(\sum_{k=0}^{\infty}\psi^0(k)e^{-i\lambda k}\rd)^*
 \ld(\sum_{k=0}^{\infty}\psi^0(k)e^{-i\lambda k}\rd),\ee
determine a solution of the constrained optimization problem
\be
 \ld\langle \overline{\psi}_{\tau}\me C_{\tau}^g,\overline{\psi}_{\tau}\me C_{\tau}^g\rd\rangle\rightarrow\inf, \label{minimax3f_f_n-n_c_fact}\ee
\[
 f(\lambda)=\Theta(e^{-i\lambda})\Theta^*(e^{-i\lambda})
 -g(\lambda)\in \mathcal{D}_f,\]
 where  $\Psi(e^{-i\lambda})\Theta(e^{-i\lambda})=E_M$. The minimax spectral characteristic $h^0=\vec h_{\tau}(f^0,g)$ is calculated by formula (\ref{spectr2 A_f_n-n_c_fact}) if
$\vec h_{\tau}(f^0,g)\in H_{\mathcal{D}}$.
\end{lema}
\begin{lema}
  The spectral density $g^0\in\mathcal{D}_g$ which admits the canonical factorizations (\ref{fakt2_f_n-n_c}), (\ref{fakt6_f_n-n_c}) with a known regular spectral density $f(\lambda)$ is the least favourable spectral density in the class
 $\md D_g$ for the optimal
linear filtering of the functional $A\xi$ based on observations of the process $\xi(t)+\eta(t)$ at points $t\leq0$ if the matrix coefficients $\{\psi^0(k): k\geq0\}$
from the canonical factorization
\be \label{fakt2_lg_f_n-n_c}
 (f(\lambda)+g^0(\lambda))^{-1} =
 \ld(\sum_{k=0}^{\infty}\psi^0(k)e^{-i\lambda k}\rd)^*
 \ld(\sum_{k=0}^{\infty}\psi^0(k)e^{-i\lambda k}\rd),
 \ee
determine a solution of the constrained optimization problem
\be
 \ld\langle\overline{\psi}_{\tau}\me C_{\tau}^f,\overline{\psi}_{\tau}\me C_{\tau}^f\rd\rangle\rightarrow\inf, \label{minimax3g_f_n-n_c_fact}\ee
\[
 g(\lambda)=
\Theta(e^{-i\lambda})\Theta^*(e^{-i\lambda})
-f(\lambda)\in \mathcal{D}_g
\]
 where  $\Psi(e^{-i\lambda})\Theta(e^{-i\lambda})=E_M$. The minimax spectral characteristic $h^0=\vec h_{\tau}(f,g^0)$ is calculated by formula (\ref{spectr A_f_n-n_c_fact}) if
$\vec h_{\tau}(f,g^0)\in H_{\mathcal{D}}$.
\end{lema}

The least favorable spectral densities $(f^0,g^0)$ and the function $h^0$ form a saddle point of the function
$\Delta(h;f,g)$ on the set $H_{\mathcal{D}}\times\mathcal{D}$.
The saddle point inequalities
\[\Delta(h;f^0,g^0)\geq\Delta(h^0;f^0,g^0)\geq\Delta(h^0;f,g)
\quad\forall f\in \mathcal{D}_f,\forall g\in \mathcal{D}_g,\forall h\in H_{\mathcal{D}}\]
 hold true if
$h^0=\vec h_{\tau}(f^0,g^0)$ and $\vec h_{\tau}(f^0,g^0)\in H_{\mathcal{D}}$,
where $(f^0,g^0)$ is a solution of the constrained optimization problem
\be
 \widetilde{\Delta}(f,g)=-\Delta(\vec h_{\tau}(f^0,g^0);f,g)\to \inf,\quad (f,g)\in \mathcal{D}.\label{zad_zum_extr_f_f_n-n_c}
\ee
Here the functional $\Delta(\vec h_{\tau}(f^0,g^0);f,g)$ is calculated by the formula
\begin{align*}
\Delta(\vec h_{\tau}(f^0,g^0);f,g)
&=\frac{1}{2\pi}\int_{-\pi}^{\pi}
\frac{\lambda^{2d}}{|1-e^{i\lambda\tau}|^{2d}}
(\me C^{f0}_{\tau}(e^{i\lambda}))^{\top}(f^0(\lambda)+g^0(\lambda))^{-1}f(\lambda)
\\
&\quad\quad\times
(f^0(\lambda)+g^0(\lambda))^{-1}
\overline{\me C^{f0}_{\tau}(e^{i\lambda})}
d\lambda
\\
&\quad+\frac{1}{2\pi}\int_{-\pi}^{\pi}
\frac{\lambda^{2d}}{|1-e^{i\lambda\tau}|^{2d}}
(\me C^{g0}_{\tau}(e^{i\lambda}))^{\top}(f^0(\lambda)+g^0(\lambda))^{-1} g(\lambda)
\\
&\quad\quad\times
(f^0(\lambda)+g^0(\lambda))^{-1}
\overline{\me C^{g0}_{\tau}(e^{i\lambda})}
d\lambda,
\end{align*}
where
\begin{align*}
(\me C^{f0}_{\tau}(e^{i\lambda}))^{\top}
&:=
(\vec B_{\tau}(e^{-i\lambda}))^{\top}\frac{|1-e^{i\lambda
 \tau}|^{2d}}{\lambda^{2d}}g^0(\lambda)+(\vec C_{\tau}^0(e^{i\lambda}))^{\top},
\\
({\me C}^{g0}_{\tau}(e^{i \lambda}))^{\top}
&:=
(\vec B_{\tau}(e^{-i\lambda}))^{\top}\frac{|1-e^{i\lambda
 \tau}|^{2d}}{\lambda^{2d}}f^0(\lambda)--(\vec C_{\tau}^0(e^{i\lambda}))^{\top}.
\end{align*}
and
\[
\vec C_{\tau}^0(e^{i\lambda})=\sum_{j=0}^{\infty}((\me
P_{\tau}^0)^{-1}\me R^0 D^{\tau}\me a)_je^{i\lambda
(j+1)},\]
or, taking into account (\ref{formula_11_f_n-n_c}), it is calculated by the formula
\begin{align*}
&\Delta(\vec h_{\tau}(f^0,g^0);f,g)=
\\&=\frac{1}{2\pi}\ip
(\me h^0_{\tau,g}(e^{-i\lambda}))^{\top}\Psi^0(e^{-i\lambda })f(\lambda)
(\Psi^0(e^{-i\lambda }))^*\overline{\me h^0_{\tau,g}(e^{-i\lambda})}d\lambda
\\
&\quad+
\frac{1}{2\pi}\ip
(\me h^0_{\tau,f}(e^{-i\lambda}))^{\top}\Psi^0(e^{-i\lambda })g(\lambda)
(\Psi^0(e^{-i\lambda }))^*\overline{\me h^0_{\tau,f}(e^{-i\lambda})}d\lambda,
\end{align*}
where
\begin{align*}
 \me h^0_{\tau,g}(e^{-i\lambda})&=\sum_{j=0}^{\infty}( \overline{\psi}^0_{\tau}(\me C^g_{\tau})^0 )_j e^{-i\lambda j}.
 \\
 \me h^0_{\tau,f}(e^{-i\lambda})&=\sum_{j=0}^{\infty}(\overline{\psi}^0_{\tau} (\me C^f_{\tau})^0 )_j e^{-i\lambda j},
 \end{align*}

The constrained optimization problem \eqref{zad_zum_extr_f_f_n-n_c} is equivalent to the unconstrained optimization problem
\be \label{zad_unconst_extr_f_st_d}
 \Delta_{\mathcal{D}}(f,g)=\widetilde{\Delta}(f,g)+ \delta(f,g|\mathcal{D})\to\inf,\ee
 where $\delta(f,g|\mathcal{D})$ is the indicator function of the set
$\mathcal{D}$, namely $\delta(f,g|\mathcal{D})=0$ if $(f,g)\in \mathcal{D}$ and $\delta(f,g|\mathcal{D})=+\infty$ if $(f,g)\notin \mathcal{D}$.
The condition
 $0\in\partial\Delta_{\mathcal{D}}(f^0,g^0)$, where
$\partial\Delta_{\mathcal{D}}(f^0,g^0)$ is the subdifferential of the functional $\Delta_{\mathcal{D}}(f,g)$ at point $(f^0,g^0)\in \mathcal{D}=\mathcal{D}_f\times\mathcal{D}_g$,
 characterizes a solution $(f^0,g^0)$ of the stated unconstrained optimization problem.
 This condition is the necessary and sufficient condition under which the point $(f^0,g^0)$ belongs to the set of minimums of the convex functional $\Delta_{\mathcal{D}}(f,g)$ \cite{Moklyachuk2015,Rockafellar}.
 Thus, it allows us to find the equations which determine the least favourable spectral densities in some special classes of spectral densities $\md D$.

The form of the functional $\Delta(\vec h_{\overline{\mu}}(f^0,g^0);f,g)$ is suitable for application of the Lagrange method of indefinite
multipliers to the constrained optimization problem \eqref{zad_zum_extr_f_f_n-n_c}.
Thus, the complexity of the problem is reduced to  finding the subdifferential of the indicator function of the set of admissible spectral densities. We illustrate the solving of the problem \eqref{zad_unconst_extr_f_st_d} for concrete sets admissible spectral densities  in the following subsections.

\subsection{Least favorable spectral densities in the class $\md D_f^0\times\md D_g^0$}
\label{subsec_set_D0_f_n-n_c}

 Consider   the minimax (robust) filtering problem  for the functional $A\xi$ in case of two uncorrelated stochastic  processes $\xi(t)$ and $\eta(t)$ with   periodically stationary increments
under the condition that the spectral densities belong to the sets of admissible spectral densities $\mathcal{D}=\md D_0^k\times\md D_{\varepsilon}^{k}$
where
$\md D_0^k$, $k=1,2,3,4$, are
\begin{align*}
 \md D_{0}^{1} &=\bigg\{f(\lambda )\bigg|\frac{1}{2\pi} \int_{-\pi}^{\pi}
f(\lambda )d\lambda  =P \bigg\},
\\
\md D_{0}^{2} &=\bigg\{f(\lambda )\bigg|\frac{1}{2\pi }
\int_{-\pi}^{\pi}
{\rm{Tr}}\,[ f(\lambda )]d\lambda =p \bigg\},
\\
\md D_{0}^{3} &=\bigg\{f(\lambda )\bigg|\frac{1}{2\pi}
\int_{-\pi}^{\pi}
f_{kk} (\lambda )d\lambda =p_{k}, k=\overline{1,\infty} \bigg\},
\\
\md D_{0}^{4} &=\bigg\{f(\lambda )\bigg|\frac{1}{2\pi} \int_{-\pi}^{\pi}
\left\langle B_{1} ,f(\lambda )\right\rangle d\lambda =p \bigg\},
\end{align*}

\noindent
where  $p, p_k, k=\overline{1,\infty}$ are given numbers, $P, B_1$ are given positive-definite Hermitian matrices, and
$\md D_{\varepsilon}^{k}$, $k=1,2,3,4$ are

\begin{align*}
\md D_{\varepsilon }^{1}  &=\bigg\{g(\lambda )\bigg|{\mathrm{Tr}}\,
[g(\lambda )]=(1-\varepsilon ) {\mathrm{Tr}}\,  [g_{1} (\lambda
)]+\varepsilon {\mathrm{Tr}}\,  [W(\lambda )],
\frac{1}{2\pi} \int_{-\pi}^{\pi}
{\mathrm{Tr}}\,
[g(\lambda )]d\lambda =q \bigg\};
\\
\md D_{\varepsilon }^{2}  &=\bigg\{g(\lambda )\bigg|g_{kk} (\lambda)
=(1-\varepsilon )g_{kk}^{1} (\lambda )+\varepsilon w_{kk}(\lambda),
\\&\quad\quad\quad\quad\quad\quad\quad\quad\quad\quad\quad\quad\quad
 \quad\quad\quad\quad\quad
\frac{1}{2\pi} \int_{-\pi}^{\pi}
g_{kk} (\lambda)d\lambda  =q_{k} , k=\overline{1,\infty}\bigg\};
\\
\md D_{\varepsilon }^{3} &=\bigg\{g(\lambda )\bigg|\left\langle B_{2},g(\lambda )\right\rangle =(1-\varepsilon )\left\langle B_{2},g_{1} (\lambda )\right\rangle+\varepsilon \left\langle B_{2},W(\lambda )\right\rangle,
\\&\quad\quad\quad\quad\quad\quad\quad\quad\quad\quad\quad\quad\quad
 \quad\quad\quad\quad\quad\quad\quad
\frac{1}{2\pi}\int_{-\pi}^{\pi}
\left\langle B_{2} ,g(\lambda )\right\rangle d\lambda =q\bigg\};
\\
\md D_{\varepsilon }^{4}&=\bigg\{g(\lambda )
\bigg|g(\lambda)=(1-\varepsilon )g_{1} (\lambda )+\varepsilon W(\lambda ),
\frac{1}{2\pi } \int_{-\pi}^{\pi}
g(\lambda )d\lambda=Q\bigg\},
\end{align*}

\noindent
where  $g_{1} ( \lambda )$ is a fixed spectral density, $W(\lambda)$ is an unknown spectral density, $q, q_k,k=\overline{1,\infty}$, are given numbers, $Q$ is a given positive-definite Hermitian matrix.

 The condition
$0\in\partial\Delta_{\md D}(f^0,g^0)$ implies that the least favourable spectral densities $f^0(\lambda)\in\md D_f^0$ and
$g^0(\lambda)\in\md D_g^0$ satisfy the equations described below for each pair of the sets of admissible spectral densities.

For the first pair of the sets of admissible spectral densities $\md D_{0}^1\times \md D_{\varepsilon} ^{1}$, we have equations
\begin{multline}  \label{eq_4_1_f_n-n_c}
\frac{\lambda^{2d}}{|1-e^{i\lambda\tau}|^{2d}}
\left(
\me C^{f0}_{\tau}(e^{i\lambda})
\right)
\left(
\me C^{f0}_{\tau}(e^{i\lambda})\right)^{*}
=
\\
=
   \ld(f^0(\lambda)+ g^0(\lambda)\rd)
\vec{\alpha}
\cdot
\vec{\alpha}^{*}
  \ld(f^0(\lambda)+ g^0(\lambda)\rd),
\end{multline}
\begin{multline}\label{eq_5_1_f_n-n_c}
\frac{\lambda^{2d}}{|1-e^{i\lambda\tau}|^{2d}}
\left(
\me C^{g0}_{\tau}(e^{i\lambda})
\right)
\left(
\me C^{g0}_{\tau}(e^{i\lambda})
\right)^{*}
=
(\alpha^{2} +\gamma_1(\lambda ))
\left(
  f^0(\lambda)+ g^0(\lambda)
\right)^2,
\end{multline}

\noindent where $\alpha^{2}$, $\vec{\alpha}$ are Lagrange multipliers,  the function $\gamma_1(\lambda )\le 0$ and $\gamma_1(\lambda )=0$ if ${\mathrm{Tr}}\,[g^{0} (\lambda )]>(1-\varepsilon ) {\mathrm{Tr}}\, [g_{1} (\lambda )]$.

For the second pair of the sets of admissible spectral densities $\md D_{f0}^2\times \md D_{\varepsilon } ^{2}$, we have equation
\begin{equation}  \label{eq_4_2_f_n-n_c}
\frac{\lambda^{2d}}{|1-e^{i\lambda\tau}|^{2d}}
\left(
\me C^{f0}_{\tau}(e^{i\lambda})
\right)
\left(
\me C^{f0}_{\tau}(e^{i\lambda})\right)^{*}
=
\alpha^{2} \left(
  f^0(\lambda)+ g^0(\lambda)
\right)^{2},
\end{equation}
\begin{multline}   \label{eq_5_2_f_n-n_c}
\frac{\lambda^{2d}}{|1-e^{i\lambda\tau}|^{2d}}
\left(
\me C^{g0}_{\tau}(e^{i\lambda})
\right)
\left(
\me C^{g0}_{\tau}(e^{i\lambda})
\right)^{*}
=
\\
=
  (f^0(\lambda)+ g^0(\lambda))
\left\{(\alpha_{k}^{2} +\gamma_{k}^1 (\lambda ))\delta _{kl} \right\}_{k,l=1}^{\infty}
  (f^0(\lambda)+ g^0(\lambda)),
\end{multline}

\noindent where $\alpha^{2}$, $\alpha _{k}^{2}$ are Lagrange multipliers,  functions $\gamma_{k}^1(\lambda )\le 0$ and $\gamma_{k}^1 (\lambda )=0$ if $
g_{kk}^{0}(\lambda )>(1-\varepsilon )g_{kk}^{1} (\lambda )$.

For the third pair of the sets of admissible spectral densities $\md D_{f0}^3\times \md D_{\varepsilon }^{3}$, we have equation
\begin{multline}   \label{eq_4_3_f_n-n_c}
\frac{\lambda^{2d}}{|1-e^{i\lambda\tau}|^{2d}}
\left(
\me C^{f0}_{\tau}(e^{i\lambda})
\right)
\left(
\me C^{f0}_{\tau}(e^{i\lambda})\right)^{*}
=
\\
=
  (f^0(\lambda)+ g^0(\lambda))
\left\{\alpha _{k}^{2} \delta _{kl} \right\}_{k,l=1}^{\infty}
  (f^0(\lambda)+ g^0(\lambda)),
\end{multline}
\begin{multline}   \label{eq_5_3_f_n-n_c}
\frac{\lambda^{2d}}{|1-e^{i\lambda\tau}|^{2d}}
\left(
\me C^{g0}_{\tau}(e^{i\lambda})
\right)
\left(
\me C^{g0}_{\tau}(e^{i\lambda})
\right)^{*}
=
\\
=
\left(\alpha^{2} +\gamma_1'(\lambda )\right)
  (f^0(\lambda)+ g^0(\lambda))B_{2}^{ \top}\,
  (f^0(\lambda)+ g^0(\lambda)),
\end{multline}

\noindent where $\alpha _{k}^{2}$, $\alpha^{2}$ are Lagrange multipliers,  function $\gamma_1' ( \lambda )\le 0$ and $\gamma_1' ( \lambda )=0$ if $\langle B_{2} ,g^{0} ( \lambda ) \rangle>(1- \varepsilon ) \langle B_{2} ,g_{1} ( \lambda ) \rangle$,
 $\delta _{kl}$ are Kronecker symbols.

For the fourth pair of the sets of admissible spectral densities $\md D_{f0}^4\times \md D_{\varepsilon}^{4}$, we have equation
\begin{multline} \label{eq_4_4_f_n-n_c}
\frac{\lambda^{2d}}{|1-e^{i\lambda\tau}|^{2d}}
\left(
\me C^{f0}_{\tau}(e^{i\lambda})
\right)
\left(
\me C^{f0}_{\tau}(e^{i\lambda})\right)^{*}
=
\\
=
\alpha^{2}   (f^0(\lambda)+ g^0(\lambda))
B_{1}^{\top}
   (f^0(\lambda)+ g^0(\lambda)),
\end{multline}
\begin{multline}  \label{eq_5_4_f_n-n_c}
\frac{\lambda^{2d}}{|1-e^{i\lambda\tau}|^{2d}}
\left(
\me C^{g0}_{\tau}(e^{i\lambda})
\right)
\left(
\me C^{g0}_{\tau}(e^{i\lambda})
\right)^{*}
=
\\
=
  (f^0(\lambda)+ g^0(\lambda))
(\vec{\alpha}\cdot \vec{\alpha}^{*}+\Gamma(\lambda))
  (f^0(\lambda)+ g^0(\lambda)),
\end{multline}

\noindent where $\alpha^{2}$, $\vec{\alpha}$ are Lagrange multipliers,  function $\Gamma(\lambda )\le 0$ and $\Gamma(\lambda )=0$ if $g^{0}(\lambda )>(1-\varepsilon )g_{1} (\lambda )$.

The following theorem  holds true.

\begin{thm}
The matrix functions
$f^0(\lambda)$ and
$g^0(\lambda)$ determined by the pairs of equations
(\ref{eq_4_1_f_n-n_c})--(\ref{eq_5_1_f_n-n_c}), (\ref{eq_4_2_f_n-n_c})--(\ref{eq_5_2_f_n-n_c}), (\ref{eq_4_3_f_n-n_c})--(\ref{eq_5_3_f_n-n_c}), (\ref{eq_4_4_f_n-n_c})--(\ref{eq_5_4_f_n-n_c})
and restrictions  on the spectral densities matrices from the corresponding classes $ \md  D_0^{k}\times{\md D_{\varepsilon }^{k}}$, $k=1,2,3,4$,
are the least
favourable in the class $ \md  D_0^{k}\times{\md D_{\varepsilon }^{k}}$, $k=1,2,3,4$, for the optimal linear
filtering of the functional $A\xi$
 if they satisfy minimality condition (\ref{umova11_f_n-n_c}) and
 determine a solution of constrained optimization problem (\ref{minimax1_f_n-n_c})
  The minimax-robust spectral characteristic $\vec h_{\tau}(f^0,g^0)$ of the optimal estimate of the functional $A{\xi}$ is determined by   formula (\ref{spectr A_f_n-n_c}).
\end{thm}

Let the spectral densities $f(\lambda)$ and $g(\lambda)$ admit the canonical factorizations (\ref{fakt1_f_n-n_c}), (\ref{fakt3_f_n-n_c}) and (\ref{fakt6_f_n-n_c}). Then equations (\ref{eq_4_1_f_n-n_c})--(\ref{eq_5_4_f_n-n_c}) for the least favourable spectral densities can be presented in terms of the coefficients of these canonical factorizations

For the   set $\md D_{0}^1\times \md D_{\varepsilon} ^{1}$  of admissible spectral densities, we have equations
\begin{align}  \label{eq_4_1_f_n-n_c_fact}
\left(\me h^0_{\tau,g}(e^{-i\lambda})\right)
\left(\me h^0_{\tau,g}(e^{-i\lambda})\right)^{*}
&=
    (\Theta^0(e^{-i\lambda}))^{\top}
\vec{\alpha}
\cdot
\vec{\alpha}^{*}
  \overline{\Theta^0(e^{-i\lambda})},
\\
\label{eq_5_1_f_n-n_c_fact}
\left(\me h^0_{\tau,f}(e^{-i\lambda})\right)
\left(\me h^0_{\tau,f}(e^{-i\lambda})\right)^{*}
&=
(\alpha^{2} +\gamma_1(\lambda ))
 (\Theta^0(e^{-i\lambda}))^{\top}
 \overline{\Theta^0(e^{-i\lambda})},
\end{align}
for the   set  $\md D_{0}^2\times \md D_{\varepsilon } ^{2}$ of admissible spectral densities, we have equation
\begin{align}  \label{eq_4_2_f_n-n_c_fact}
\left(\me h^0_{\tau,g}(e^{-i\lambda})\right)
\left(\me h^0_{\tau,g}(e^{-i\lambda})\right)^{*}
&=
\alpha^{2}  (\Theta^0(e^{-i\lambda}))^{\top}
\overline{\Theta^0(e^{-i\lambda})},
\\   \label{eq_5_2_f_n-n_c_fact}
\left(\me h^0_{\tau,f}(e^{-i\lambda})\right)
\left(\me h^0_{\tau,f}(e^{-i\lambda})\right)^{*}
&=
   (\Theta^0(e^{-i\lambda}))^{\top}
\left\{(\alpha_{k}^{2} +\gamma_{k}^1 (\lambda ))\delta _{kl} \right\}_{k,l=1}^{\infty}
  \overline{\Theta^0(e^{-i\lambda})},
\end{align}
for the   set $\md D_{0}^3\times \md D_{\varepsilon }^{3}$ of admissible spectral densities, we have equation
\begin{align}   \label{eq_4_3_f_n-n_c_fact}
\left(\me h^0_{\tau,g}(e^{-i\lambda})\right)
\left(\me h^0_{\tau,g}(e^{-i\lambda})\right)^{*}
&=
   (\Theta^0(e^{-i\lambda}))^{\top}
\left\{\alpha _{k}^{2} \delta _{kl} \right\}_{k,l=1}^{\infty}
  \overline{\Theta^0(e^{-i\lambda})},
\\   \label{eq_5_3_f_n-n_c_fact}
\left(\me h^0_{\tau,f}(e^{-i\lambda})\right)
\left(\me h^0_{\tau,f}(e^{-i\lambda})\right)^{*}
&=
\left(\alpha^{2} +\gamma_1'(\lambda )\right)
   (\Theta^0(e^{-i\lambda}))^{\top}B_{2}^{ \top}\,
  \overline{\Theta^0(e^{-i\lambda})},
\end{align}
for the   set $\md D_{0}^4\times \md D_{\varepsilon}^{4}$ of admissible spectral densities, we have equation
\begin{align} \label{eq_4_4_f_n-n_c_fact}
\left(\me h^0_{\tau,g}(e^{-i\lambda})\right)
\left(\me h^0_{\tau,g}(e^{-i\lambda})\right)^{*}
&=
\alpha^{2}    (\Theta^0(e^{-i\lambda}))^{\top}
B_{1}^{\top}
   \overline{\Theta^0(e^{-i\lambda})},
\\  \label{eq_5_4_f_n-n_c_fact}
\left(\me h^0_{\tau,f}(e^{-i\lambda})\right)
\left(\me h^0_{\tau,f}(e^{-i\lambda})\right)^{*}
&=
 (\Theta^0(e^{-i\lambda}))^{\top}
(\vec{\alpha}\cdot \vec{\alpha}^{*}+\Gamma(\lambda))
  \overline{\Theta^0(e^{-i\lambda})},
\end{align}

\noindent where the Lagrange multipliers $\alpha^{2}$, $\vec{\alpha}$,   $\alpha _{k}^{2}$, and  the function $\gamma_1(\lambda)$, $\gamma_{k}^1(\lambda)$, $\gamma_1'(\lambda)$, $\Gamma(\lambda)$ are defined in the same way as in equations (\ref{eq_4_1_f_n-n_c})--(\ref{eq_5_4_f_n-n_c}; $\delta _{kl}$ are Kronecker symbols.

The following theorem  holds true.

\begin{thm}
The matrix functions
$f^0(\lambda)$ and $g^0(\lambda)$ determined
by  canonical factorizations   (\ref{fakt231_f_n-n_c}), (\ref{fakt232_f_n-n_c}), (\ref{fakt233_f_n-n_c}) ,
 pairs of equations
(\ref{eq_4_1_f_n-n_c_fact})--(\ref{eq_5_1_f_n-n_c_fact}), (\ref{eq_4_2_f_n-n_c_fact})--(\ref{eq_5_2_f_n-n_c_fact}), (\ref{eq_4_3_f_n-n_c_fact})--(\ref{eq_5_3_f_n-n_c_fact}), (\ref{eq_4_4_f_n-n_c_fact})--(\ref{eq_5_4_f_n-n_c_fact})
and restrictions  on the spectral densities matrices from the corresponding classes $ \md  D_0^{k}\times{\md D_{\varepsilon }^{k}}$, $k=1,2,3,4$,
are the least
favourable in the class $ \md  D_0^{k}\times{\md D_{\varepsilon }^{k}}$, $k=1,2,3,4$, for the optimal linear
filtering of the functional $A\xi$
 if they satisfy minimality condition (\ref{umova11_f_n-n_c}) and determine a solution of the constrained optimization problem (\ref{minimax1f_f_n-n_c_fact}) (or  (\ref{minimax1g_f_n-n_c_fact})).
  The minimax-robust spectral characteristic $\vec h_{\tau}(f^0,g^0)$ of the optimal estimate of the functional $A{\xi}$ is determined by   formula (\ref{spectr A_f_n-n_c_fact}) (or (\ref{spectr2 A_f_n-n_c_fact})).
\end{thm}

\subsection{Semi-uncertain filtering problem in classes $\md D_0$ of least favorable spectral density}

Consider   the minimax (robust) filtering problem  for the functional $A\xi$ when the spectral density $g(\lambda)$ is known and the spectral density $f(\lambda)$ belongs to the sets of admissible spectral densities
$\md D_0^k$, $k=1,2,3,4$, defined in Subsection \ref{subsec_set_D0_f_n-n_c}.
Let   the spectral densities $f^0(\lambda)$ and $g(\lambda)$ admit the canonical factorizations (\ref{fakt1_f_n-n_c}), (\ref{fakt3_f_n-n_c}) and (\ref{fakt6_f_n-n_c}). In this case we obtain the following equations  for the least favourable spectral densities  in terms of the coefficients of these canonical factorizations:
\begin{multline}  \label{eq_4_1_semi_f_n-n_c_fact}
\left(
\sum_{j=0}^{\infty}( \overline{\psi}^0_{\tau}\me C^g_{\tau})_j e^{-i\lambda j}
\right)
\left(
\sum_{j=0}^{\infty}( \overline{\psi}^0_{\tau}\me C^g_{\tau})_j e^{-i\lambda j}
\right)^{*}
=
\\
=
    \left(\sum_{k=0}^{\infty}\theta^0(k)e^{-i\lambda k}\right)^{\top}
\vec{\alpha}
\cdot
\vec{\alpha}^{*}
  \overline{\left(\sum_{k=0}^{\infty}\theta^0(k)e^{-i\lambda k}\right)},
\end{multline}
\begin{multline}  \label{eq_4_2_semi_f_n-n_c_fact}
\left(
\sum_{j=0}^{\infty}( \overline{\psi}^0_{\tau}\me C^g_{\tau})_j e^{-i\lambda j}
\right)
\left(
\sum_{j=0}^{\infty}( \overline{\psi}^0_{\tau}\me C^g_{\tau})_j e^{-i\lambda j}
\right)^{*}
=
\\
=
\alpha^{2}  \left(\sum_{k=0}^{\infty}\theta^0(k)e^{-i\lambda k}\right)^{\top}
\overline{\left(\sum_{k=0}^{\infty}\theta^0(k)e^{-i\lambda k}\right)},
\end{multline}
\begin{multline}   \label{eq_4_3_semi_f_n-n_c_fact}
\left(
\sum_{j=0}^{\infty}( \overline{\psi}^0_{\tau}\me C^g_{\tau})_j e^{-i\lambda j}
\right)
\left(
\sum_{j=0}^{\infty}( \overline{\psi}^0_{\tau}\me C^g_{\tau})_j e^{-i\lambda j}
\right)^{*}
=
\\
=
   \left(\sum_{k=0}^{\infty}\theta^0(k)e^{-i\lambda k}\right)^{\top}
\left\{\alpha _{k}^{2} \delta _{kl} \right\}_{k,l=1}^{\infty}
  \overline{\left(\sum_{k=0}^{\infty}\theta^0(k)e^{-i\lambda k}\right)},
\end{multline}
\begin{multline} \label{eq_4_4_semi_f_n-n_c_fact}
\left(
\sum_{j=0}^{\infty}( \overline{\psi}^0_{\tau}\me C^g_{\tau})_j e^{-i\lambda j}
\right)
\left(
\sum_{j=0}^{\infty}( \overline{\psi}^0_{\tau}\me C^g_{\tau})_j e^{-i\lambda j}
\right)^{*}
=
\\
=
\alpha^{2}    \left(\sum_{k=0}^{\infty}\theta^0(k)e^{-i\lambda k}\right)^{\top}
B_{1}^{\top}
   \overline{\left(\sum_{k=0}^{\infty}\theta^0(k)e^{-i\lambda k}\right)},
\end{multline}

\noindent where $\alpha^{2}$, $\vec{\alpha}$,   $\alpha _{k}^{2}$ are the Lagrange multipliers,   the matrix $\me C^g_{\tau}$ is known and defined bby the coefficients of the canonical factorization of the spectral density $g(\lambda)$.

The following theorem  holds true.

\begin{thm}
 Suppose that the   spectral density $g(\lambda)$ is known and allow the canonical factorization (\ref{fakt6_f_n-n_c}).
The matrix function
$f^0(\lambda)$   determined by
   canonical factorizations   (\ref{fakt231_f_n-n_c}), (\ref{fakt232_f_n-n_c}) or (\ref{fakt233_f_n-n_c})
equations
(\ref{eq_4_1_semi_f_n-n_c_fact}), (\ref{eq_4_2_semi_f_n-n_c_fact}), (\ref{eq_4_3_semi_f_n-n_c_fact}), (\ref{eq_4_4_semi_f_n-n_c_fact})
and   restrictions on
the spectral density matrix from the corresponding classes $ \md  D_0^{k}$, $k=1,2,3,4$,
is
the least favourable in the class $ \md  D_0^{k}$, $k=1,2,3,4$, for the optimal linear
filtering of the functional $A\xi$ if
the function $f^0(\lambda)+g(\lambda)$ satisfies  minimality condition   (\ref{umova11_f_n-n_c}) and determines solution of the constrained optimization
problem (\ref{minimax3f_f_n-n_c_fact}).
  Coefficients $\{\psi^0_{\tau}(k):k\geq0\}$ and $\{\theta^0(k):k\geq0\}$ are connected by   relations  $\Psi^0(e^{-i\lambda})\Theta^0(e^{-i\lambda})=E_M$ and   (\ref{formula_11_f_n-n_c}). The minimax-robust spectral characteristic $\vec h_{\tau}(f^0,g)$ of the optimal estimate of the functional $A{\xi}$ is determined by   formula (\ref{spectr2 A_f_n-n_c_fact}).
\end{thm}

\section*{Appendix}
\emph{\textbf{Proof of Theorem \ref{thm1_f_n-n_c}}}.

The projection $\widehat{B}\vec \xi$  is characterized by the following
conditions:

1) $ \widehat{B}\vec \xi\in H^{0}(\vec\xi^{\,(d)}+\vec\eta^{\,(d)})$;

2) $(B\vec\xi-\widehat{B} \vec\xi)\perp
H^{0}(\vec\xi^{\,(d)}+\vec\eta^{\,(d)})$.
Condition 2) implies that  for all $j\leq0$ and $k\geq1$ the function
$\vec h_{\tau}(\lambda)$ satisfy the relation
\begin{multline*}
 \mt E(B\vec\xi-\widehat{B}\vec\xi)^{\top}
 (\overline{\xi^{(d)}_k(j,\tau)+\eta^{(d)}_k(j,\tau)})=
 \\
 =
 \ip\ld((\vec B_{\tau}(e^{-i\lambda }))^{\top}
 \frac{(1-e^{-i\lambda\tau})^{d}}{(i\lambda)^{d}}f(\lambda)--(\vec h_{\tau}(\lambda))^{\top}(f(\lambda)+g(\lambda))\rd)
 \times\\
 \times\frac{(1-e^{i\lambda\tau})^d}{(-i\lambda)^{d}}e^{-i\lambda j}\vec\delta_{k}d\lambda=0.
 \end{multline*}
Thus, the spectral characteristic $\vec  h_{\tau}(\lambda)$ allows a representation
\begin{align*}
  (\vec h_{\tau}(\lambda))^{\top}&=
 (\vec B_{\tau}(e^{-i\lambda }))^{\top}\frac{(1-e^{-i\lambda\tau})^d}{(i\lambda)^{d}}f(\lambda)(f(\lambda)+g(\lambda))^{-1}
 \\
 &\quad-\frac{(-i\lambda)^{d}
}{(1-e^{i\lambda\tau})^d}(\vec C_{\tau}(e^{i\lambda}))^{\top}(f(\lambda)+g(\lambda))^{-1},
\end{align*}
 where
 \[
 \vec C_{\tau}(e^{i\lambda})=\sum_{j=0}^{\infty}\vec c^{\,\tau}_je^{i\lambda
 (j+1)},\]
 and $\vec{c}^{\,\tau}_j=\{c^{\,\tau}_{kj}\}_{k=1}^{\infty}, j\geq 0$  are unknown vectors to be found.

 Condition 1) implies that the function $\vec h_{\tau}(\lambda)$ is of the form
\[
\vec{h}_{\tau}(\lambda)=\vec{h} (\lambda)(1-e^{-i\lambda\tau})^d
\frac{1}{(i\lambda)^d}, \quad \vec{h} (\lambda)=
\sum_{j=0}^{\infty}\vec{s}_je^{-i\lambda j}.
\]
Thus, for all $l\geq 1$
\be
 \ip \ld((\vec B_{\tau}(e^{-i\lambda}))^{\top}f(\lambda)
 -\frac{\lambda^{2d} }{|1-e^{i\lambda\tau}|^{2d}}(\vec C_{\tau}(e^{i\lambda}))^{\top}\rd)(f(\lambda)+g(\lambda))^{-1}
 e^{-i\lambda l}d\lambda=0.\label{spivv_um_1_f_n-n_c}\ee

One can conclude that relation (\ref{spivv_um_1_f_n-n_c}) is equivalent to the following system of linear equations
\[
 \sum_{m=0}^{\infty} R_{l+1,m}\vec b^{\tau}_m=\sum_{j=0}^{\infty} P_{l+1,j+1}^{\tau}\vec c^{\,\tau}_j,\quad l\geq1.\]
which can be written in a matrix form
\[
 \me R\me b^{\tau}=\me P_{\tau}\me c_{\tau},\]
where
\[
   \me c^{\,\tau} =((c^{\,\tau}_1)^{\top},(c^{\tau}_2)^{\top},(c^{\,\tau}_3)^{\top}, \ldots)^{\top}.
   \]
Therefore, the   coefficients $\vec{c}^{\,\tau}_j$, $j\geq0$, are calculated by the formula
\[
 \vec c^{\,\tau}_j=(\me P_{\tau}^{-1}\me R\me b^{\tau})_j,\]
where $(\me P_{\tau}^{-1}\me R\me b^{\tau})_j$ is the
 $j$th entry of the vector $\me P_{\tau}^{-1}\me R\me b^{\tau}$.

The derived expressions justify the formulas (\ref{spectr A_f_n-n_c}) and (\ref{poh A_f_n-n_c})  for calculating the spectral characteristic $\vec{h}_{\tau}(\lambda)$ and
the MSE
$\Delta(f,g;\widehat{A}\xi)$
of the estimate $\widehat{A}\xi$  of the functional $A\xi$.

\

\emph{\textbf{Proof of Lemma \ref{lema_fact_3_f_n-c_c}.}}

  Factorizations (\ref{fakt3_f_n-n_c}) and (\ref{fakt2_f_n-n_c})
\begin{align*}
\sum_{j\in\mr Z} \vec s(j)e^{i\lambda j}&=
\ld(f(\lambda)(f(\lambda)+g(\lambda))^{-1}\rd)^{\top}
\\
&=(\Psi_{\tau}(e^{-i\lambda}))^{\top}\overline{\Psi}_{\tau}(e^{-i\lambda})
\frac{|1-e^{i\lambda\tau}|^{2d}}{\lambda^{2d}}\overline{f}(\lambda)
\\
&=\sum_{l=0}^{\infty}\psi_{\tau}^{\top}(l)e^{-i\lambda l} \sum_{j\in\mr Z}Z_{\tau}(j)e^{i\lambda j}
\\
&=\sum_{j\in\mr Z} \sum_{l=0}^{\infty}\psi_{\tau}^{\top}(l)Z_{\tau}(l+j)e^{i\lambda j}.
\end{align*}
Then
\begin{align*}
 (\Theta^{\top}_{\tau}\me R\me b_{\tau})_m&=\sum_{j=0}^{\infty}\sum_{p=m}^{\infty}\theta^{\top}_{\tau}(p-m)R(p+j+1)
 \vec b^{\tau}_j
 \\&=
 \sum_{j=0}^{\infty}\sum_{p=m}^{\infty}\sum_{l=0}^{\infty}
\theta^{\top}_{\tau}(p-m)\psi_{\tau}^{\top}(l)Z_{\tau}(l+p+j+1)\vec b^{\tau}_j
\\&=
 \sum_{j=0}^{\infty}\sum_{p=m}^{\infty}\sum_{l=p}^{\infty}
(\psi_{\tau}(l-p)\theta_{\tau}(p-m))^{\top}Z_{\tau}(l+j+1)\vec b^{\tau}_j
\\&=
 \sum_{j=0}^{\infty}\sum_{l=m}^{\infty}\mt{diag}_M(\delta_{l,m})Z_{\tau}(l+j+1)\vec b^{\tau}_j
\\&=
 \sum_{j=0}^{\infty}Z_{\tau}(m+j+1)\vec b^{\tau}_j.
 \end{align*}
A representation for  $Z_{\tau}(j)$ follows from
\[
\sum_{j\in\mr Z}Z_{\tau}(j)e^{i\lambda j}
=\overline{\Psi}_{\tau}(e^{-i\lambda})\frac{|1-e^{i\lambda\tau}|^{2d}}{\lambda^{2d}}\overline{f}(\lambda)
=\sum_{j\in\mr Z} \sum_{l=0}^{\infty}\overline{\Psi}_{\tau}(l)\overline{f}_{\tau}(l-j)e^{i\lambda j}.
\]
On the other hand,
\begin{align*}
\sum_{j\in\mr Z}\vec s(k)e^{i\lambda j}&=
\ld(f(\lambda)(f(\lambda)+g(\lambda))^{-1}\rd)^{\top}
\\
&=
\ld((f(\lambda)+g(\lambda)-g(\lambda))(f(\lambda)+g(\lambda))^{-1}\rd)^{\top}
\\
&=
\ld(E_{\infty}-g(\lambda)(f(\lambda)+g(\lambda))^{-1}\rd)^{\top}
\\
&=E_{\infty}-(\Psi_{\tau}(e^{-i\lambda}))^{\top}\overline{\Psi}_{\tau}(e^{-i\lambda})
\frac{|1-e^{i\lambda\tau}|^{2d}}{\lambda^{2d}}\overline{g}(\lambda)
\\
&=E_{\infty}-\sum_{l=0}^{\infty}\psi_{\tau}^{\top}(l)e^{-i\lambda l} \sum_{j\in\mr Z}X_{\tau}(j)e^{i\lambda j}
\\
&=E_{\infty}-\sum_{j\in\mr Z} \sum_{l=0}^{\infty}\psi_{\tau}^{\top}(l)X_{\tau}(l+j)e^{i\lambda j}
\\&=
 \sum_{j=0}^{\infty}Z_{\tau}(m+j+1)\vec b^{\tau}_j,
\end{align*}
where\[
\sum_{j\in\mr Z}X_{\tau}(j)e^{i\lambda j}
=\overline{\Psi}_{\tau}(e^{-i\lambda})\frac{|1-e^{i\lambda\tau}|^{2d}}{\lambda^{2d}}
\overline{g}(\lambda)
=\sum_{k\in\mr Z} \sum_{l=0}^{\infty}\overline{\Psi}_{\tau}(l)\overline{g}_{\tau}(l-k)e^{i\lambda k}.
\]
Thus, coefficients $Z_{\tau}(j)$ and $X_{\tau}(j)$  are related as
\[\left\{
\begin{array}{rcl}
    E_{\infty}-\sum\limits_{l=0}^{\infty}\psi_{\tau}^{\top}(l)X_{\tau}(l)&=&\sum\limits_{l=0}^{\infty}\psi_{\tau}^{\top}(l)Z_{\tau}(l)
\\
    -\sum\limits_{l=0}^{\infty}\psi_{\tau}^{\top}(l)X_{\tau}(l+k)&=&\sum\limits_{l=0}^{\infty}\psi_{\tau}^{\top}(l)Z_{\tau}(l+k),\quad k\neq 0,
\end{array}
\right.
\]
which implies
\begin{align*}
 (\Theta^{\top}_{\tau}\me R\me b_{\tau})_m&=
 \sum_{j=0}^{\infty}\sum_{p=m}^{\infty}\sum_{l=0}^{\infty}
\theta^{\top}_{\tau}(p-m)\psi_{\tau}^{\top}(l)Z_{\tau}(l+p+j+1)\vec b^{\tau}_j
\\
&=-
 \sum_{j=0}^{\infty}\sum_{p=m}^{\infty}\sum_{l=0}^{\infty}
\theta^{\top}_{\tau}(p-m)\psi_{\tau}^{\top}(l)X_{\tau}(l+p+j+1)\vec b^{\tau}_j
\\&=-
 \sum_{j=0}^{\infty}X_{\tau}(m+j+1)\vec b^{\tau}_j.
 \end{align*}

 \

 \emph{\textbf{Proof of Theorem \ref{thm3_f_n-n_c_fact}.}}

Making use of Lemma   $\ref{lema_fact_3_f_n-c_c}$ rewrite formulas (\ref{spectr A_f_n-n_c}) and (\ref{poh A_f_n-n_c}). Make the following transformations:
\begin{align*}
&\notag\frac{\lambda^{2d}}{|1-e^{i\lambda\tau}|^{2d}}
\ld((f(\lambda)+g(\lambda))^{-1}\rd)^{\top}\ld(\sum_{k=0}^{\infty} (\me
 P_{\tau}^{-1}\me R\me b^{\tau})_ke^{i\lambda
 (k+1)}\rd)
 \\&= \ld(\sum_{k=0}^{\infty}\psi_{\tau}^{\top}(k)e^{-i\lambda k}\rd)\sum_{j=0}^{\infty}\sum_{k=0}^{\infty}
 \overline{\psi}_{\tau}(j)(\overline{\Theta}_{\tau}\widetilde{\me e}_{\tau})_ke^{i\lambda(k+j+1)}
 \\&= \ld(\sum_{k=0}^{\infty}\psi_{\tau}^{\top}(k)e^{-i\lambda k}\rd)\sum_{m=0}^{\infty}\sum_{p=0}^{m} \sum_{k=p}^m\overline{\psi}_{\tau}(m-k)\overline{\theta}_{\tau}(k-p)\widetilde{e}_{\tau}(p)e^{i\lambda (m+1)}
\\&= \ld(\sum_{k=0}^{\infty}\psi_{\overline{\tau}}^{\top}(k)e^{-i\lambda k}\rd)\sum_{m=0}^{\infty}\sum_{p=0}^{m}\mt{diag}_M(\delta_{m,p})\widetilde{e}_{\tau}(m)e^{i\lambda(m+1)}
 \\&= \ld(\sum_{k=0}^{\infty}\psi_{\overline{\tau}}^{\top}(k)e^{-i\lambda k}\rd)\sum_{m=0}^{\infty}\widetilde{e}_{\tau}(m)e^{i\lambda(m+1)},
 \end{align*}
and
\begin{align}
  \notag &
\ld((f(\lambda)+g(\lambda))^{-1}\rd)^{\top}(f(\lambda))^{\top}
\vec B_{\tau}(e^{-i\lambda }) =
 \\
\notag &= \Psi_{\tau}^{\top}(e^{-i\lambda})\overline{\Psi_{\tau}(e^{-i\lambda})}
\frac{|1-e^{i\lambda\tau}|^{2d}}{\lambda^{2d}}\overline{f(\lambda)}
B_{\tau}(e^{-i\lambda })
 \\
 \notag &=\ld(\sum_{k=0}^{\infty}\psi^{\top}_{\tau}(k)e^{-i\lambda k}\rd)
 \sum_{m\in \mr Z}^{\infty}
\sum_{j=0}^{\infty}Z_{\tau}(m+j)\vec b^{\tau}_je^{i\lambda m}.\label{simple_sp_char_part2_f_n_d}
\end{align}
 Then   obtain:
\begin{align*}
\notag \vec{h}_{\tau}(\lambda)
&=\frac{(1-e^{-i\lambda\tau})^d}{(i\lambda)^{d}}
\ld(\sum_{k=0}^{\infty}\psi^{\top}_{\tau}(k)e^{-i\lambda k}\rd)
\sum_{m=0}^{\infty}\sum_{j=0}^{\infty}Z_{\tau}(j-m)\vec b^{\tau}_je^{-i\lambda m}
\\
\notag &=\frac{(1-e^{-i\lambda\tau})^d}{(i\lambda)^{d}}
\ld(\sum_{k=0}^{\infty}\psi^{\top}_{\tau}(k)e^{-i\lambda k}\rd)
\\
\notag &\quad \times\sum_{m=0}^{\infty}\sum_{j=0}^{\infty}\sum_{p=m}^{\infty}
\overline{\psi}_{\tau}(p-m)\overline{f}_{\tau}(p-j)\vec b^{\tau}_je^{-i\lambda m}
\\
\notag &=\frac{(1-e^{-i\lambda\tau})^d}{(i\lambda)^{d}}
\ld(\sum_{k=0}^{\infty}\psi^{\top}_{\tau}(k)e^{-i\lambda k}\rd)
\sum_{m=0}^{\infty}( (\widetilde{\Psi}_{\tau})^*\me G^{f}_{\tau}\me b^{\tau})_m e^{-i\lambda m}
\\
 &=\frac{(1-e^{-i\lambda\tau})^d}{(i\lambda)^{d}}
\ld(\sum_{k=0}^{\infty}\psi^{\top}_{\tau}(k)e^{-i\lambda k}\rd)
\sum_{m=0}^{\infty}(\overline{\psi}_{\tau} \me C^{f}_{\tau,})_m e^{-i\lambda m}.
\end{align*}
Here
\[
(\overline{\psi}_{\tau} \me C^{f}_{\tau} )_m=\sum_{k=0}^{\infty}\overline{\psi}_{\tau}(k)\me c^{f}_{\tau}(k+m),
\]
\begin{align*}
  \me c^{f}_{\tau}(m)&=\sum_{k=0}^{\infty}\overline{f}_{\tau}(m-k)\vec b^{\tau}_k
=\sum_{l=0}^{\infty}\overline{\phi}_{\tau}(l)\sum_{k=0}^{l+m}\phi_{\tau}^{\top}(l+m-k)\vec b^{\tau}_k
\\
&=\sum_{l=0}^{\infty}\overline{\phi}_{\tau}(l)(\widetilde{\Phi}_{\tau}\me  b^{\tau})_{l+m}=(\widetilde{\Phi}_{\tau}^*\widetilde{\Phi}_{\tau}\me b^{\tau})_{m}.
\end{align*}
The MSE $\Delta(f,g;\widehat{A}\xi)$ is calculated by the formula
\begin{align*}
\notag \Delta(f,g;\widehat{A}\xi)&=\Delta(f,g;\widehat{B}\vec\xi)= \mt E|B\vec\xi-\widehat{B}\vec\xi|^2
 \\
 \notag &= \frac{1}{2\pi}\int_{-\pi}^{\pi}(\vec B_{\tau}(e^{-i\lambda}))^{\top}\frac{|1-e^{i\lambda\tau}|^{2d}}{\lambda^{2d}}f(\lambda)
 \overline{\vec B_{\tau}(e^{-i\lambda})}d\lambda
  \\
 \notag & \quad+
 \frac{1}{2\pi}\int_{-\pi}^{\pi}(\vec h_{\tau}(\lambda))^{\top}(f(\lambda)+g(\lambda))
 \overline{\vec h_{\tau}(\lambda)}d\lambda
 \\
 \notag &\quad -\frac{1}{2\pi}\int_{-\pi}^{\pi}(\vec h_{\tau}(\lambda))^{\top}
 \frac{(1-e^{i\lambda\tau})^d}{(-i\lambda)^{d}}
 \overline{\vec B_{\tau}(e^{-i\lambda})}d\lambda
  \\
 \notag &\quad-
 \frac{1}{2\pi}\int_{-\pi}^{\pi}(\vec B_{\tau}(e^{-i\lambda}))^{\top}
  \frac{(1-e^{-i\lambda\tau})^d}{(i\lambda)^{d}}f(\lambda)
  \overline{\vec h_{\tau}(\lambda)}d\lambda
 \\
 &=\|\widetilde{\Phi}_{\tau}\me b^{\tau}\|^2-\|\overline{\psi}_{\tau} \me C^{f}_{\tau}\|^2.
 \end{align*}
  Suppose that factorizations (\ref{fakt1_f_n-n_c}) and (\ref{fakt6_f_n-n_c})
   take place. Then from Lemma \ref{lema_fact_3_f_n-c_c} obtain
   \begin{multline*}
     \frac{\lambda^{2d}}{|1-e^{i\lambda\tau}|^{2d}}
\ld((f(\lambda)+g(\lambda))^{-1}\rd)^{\top}\vec C_{\tau}(e^{i\lambda})
 =  \\
  =-\ld(\sum_{k=0}^{\infty}\psi_{\overline{\tau}}^{\top}(k)e^{-i\lambda k}\rd)\sum_{m=0}^{\infty}
 \sum_{j=0}^{\infty}X_{\tau}(m+j+1)\vec b^{\tau}_je^{i\lambda(m+1)}.
   \end{multline*}
  Next,
   \begin{align}
  \notag &
\ld((f(\lambda)+g(\lambda))^{-1}\rd)^{\top}(f(\lambda))^{\top}
\vec B_{\tau}(e^{-i\lambda })
\\
\notag &=\vec B_{\tau}(e^{-i\lambda })- \ld((f(\lambda)+g(\lambda))^{-1}\rd)^{\top}(g(\lambda))^{\top}
\vec B_{\tau}(e^{-i\lambda })
 \\
\notag &=\vec B_{\tau}(e^{-i\lambda })-  \Psi_{\tau}^{\top}(e^{-i\lambda})\overline{\Psi_{\tau}(e^{-i\lambda})}
\frac{|1-e^{i\lambda\tau}|^{2d}}{\lambda^{2d}}\overline{g(\lambda)}
\vec B_{\tau}(e^{-i\lambda })
 \\
 \notag &=\vec B_{\tau}(e^{-i\lambda })- \ld(\sum_{k=0}^{\infty}\psi^{\top}_{\tau}(k)e^{-i\lambda k}\rd)
 \sum_{m\in \mr Z}^{\infty}
\sum_{j=0}^{\infty}X_{\tau}(m+j)\vec b^{\tau}_je^{i\lambda m}.\label{simple_sp_char_part2_f_n_d}
\end{align}
Finally,
 \begin{align*}
\notag \vec{h}_{\tau}(\lambda)
&=\frac{(1-e^{-i\lambda\tau})^d}{(i\lambda)^{d}}\ld(\vec B_{\tau}(e^{-i\lambda })
-
\Psi_{\tau}^{\top}(e^{-i\lambda})
\sum_{m=0}^{\infty}\sum_{j=0}^{\infty}X_{\tau}(j-m)\vec b^{\tau}_je^{-i\lambda m}\rd)
\\
\notag &=\frac{(1-e^{-i\lambda\tau})^d}{(i\lambda)^{d}}\ld(\vec B_{\tau}(e^{-i\lambda })
-
\Psi_{\tau}^{\top}(e^{-i\lambda})
\sum_{m=0}^{\infty}( (\widetilde{\Psi}_{\tau})^*\me G^{g}_{\tau}\me b^{\tau})_m e^{-i\lambda m}\rd)
\\
 &=\frac{(1-e^{-i\lambda\tau})^d}{(i\lambda)^{d}}\ld(\vec B_{\tau}(e^{-i\lambda })
-
\Psi_{\tau}^{\top}(e^{-i\lambda})
\sum_{m=0}^{\infty}(\overline{\psi}_{\tau} \me C^{g}_{\tau})_m e^{-i\lambda m}\rd),
\end{align*}
 where
\[
(\overline{\psi}_{\tau} \me C^{g}_{\tau} )_m=\sum_{k=0}^{\infty}\overline{\psi}_{\tau}(k)\me c^{g}_{\tau}(k+m),
\]
\begin{align*}
  \me c^{g}_{\tau}(m)&=\sum_{k=0}^{\infty}\overline{g}_{\tau}(m-k)\vec b^{\tau}_k
=\sum_{l=0}^{\infty}\overline{\omega}_{\tau}(l)\sum_{k=0}^{l+m}
\omega_{\tau}^{\top}(l+m-k)\vec b^{\tau}_k
\\
&=\sum_{l=0}^{\infty}\overline{\omega}_{\tau}(l)(\widetilde{\Omega}_{\tau}\me  b^{\tau})_{l+m}=(\widetilde{\Omega}_{\tau}^*\widetilde{\Omega}_{\tau}\me b^{\tau})_{m}.
\end{align*}
The  MSE is calculated by the formula
\[
 \Delta(f,g;\widehat{A}\xi)=\|\widetilde{\Omega}_{\tau} \me b^{\tau}\|^2-\|\overline{\psi}_{\tau} \me C^{g}_{\tau}\|^2.
\]

\end{document}